\documentclass[12pt]{amsart}
\usepackage{bbm}
\usepackage[pdftex]{graphicx}
\usepackage{amscd, amssymb, dsfont, upgreek, mathrsfs}
\input{xy}
\xyoption{all}

\newtheorem{Thm}{Theorem}

\newtheorem{Prop}[Thm]{Proposition}
\newtheorem{Def}[Thm]{Definition}
\newtheorem{Def/Thm}[Thm]{Definition/Theorem}
\newtheorem{Cor}[Thm]{Corollary}
\newtheorem{Lemma}[Thm]{Lemma}

\newtheorem{Rmk}[Thm]{Remark}

\newcommand{\F}{{\mathsf{M}}}

\newcommand{\ot }{\otimes}
\newcommand{\ra }{\rightarrow}

\newcommand{\lra }{\longrightarrow}

\newcommand{\Spec}{{\mathrm{Spec}}}

\newcommand{\lann}{\langle\langle}
\newcommand{\rann}{\rangle\rangle}
\newcommand{\lannn}{\left\langle\left\langle}
\newcommand{\rannn}{\right\rangle\right\rangle}

\newcommand{\G}{{\bf G}}

\newcommand{\PP }{{\mathbb P}}

\newcommand{\QQ }{{\mathbb Q}}
\newcommand{\CC }{{\mathbb C}}
\newcommand{\ZZ }{{\mathbb Z}}

\newcommand{\vir}{\mathrm{vir}}

\newcommand{\DD}{\mathsf{D}}

\newcommand{\T}{{\mathsf{T}}}

\newcommand{\lan}{\langle}
\newcommand{\ran}{\rangle}

\newcommand{\pP}{\mathsf{P}}

\newcommand{\ppl}{{\mathsf{P}}\left[}
\newcommand{\ppr}{\right]}

\def \proj {{\mathbb{P}}}
\newcommand{\com}{{\mathbb{C}}}

\begin{document}

\title[Holomorphic anomaly equations for the formal quintic]
{Holomorphic anomaly equations for \\the formal quintic}

\author{Hyenho Lho}
\address{Department of Mathematics, ETH Z\"urich}
\email {hyenho.lho@math.ethz.ch}
\author{Rahul Pandharipande}
\address{Department of Mathematics, ETH Z\"urich}
\email {rahul@math.ethz.ch}
\date{November 2018}.

\begin{abstract} 
We define a formal Gromov-Witten theory of the quintic 3-fold
via localization on $\proj^4$.
Our main result  is a direct geometric proof of holomorphic anomaly equations
for the formal quintic
in precisely the same form as predicted by B-model physics for the true 
Gromov-Witten theory of the quintic 3-fold. The results suggest that the
 formal quintic and the true quintic theories should be related by transformations which respect the holomorphic anomaly equations. Such a relationship has been
 recently found by Q. Chen, S. Guo, F. Janda, and Y. Ruan via the geometry
 of new moduli spaces.

 \end{abstract}
\maketitle

\setcounter{tocdepth}{1} 
\tableofcontents

\setcounter{section}{-1}

\section{Introduction}

\subsection{GW/SQ}\label{holp} 

Let 
$X_5\subset \mathbb{P}^4$
be a nonsingular quintic Calabi-Yau 3-fold.
The moduli space
of stable maps to the quintic
of genus $g$ and degree $d$,
$$\overline{M}_g(X_5,d)\subset \overline{M}_g(\mathbb{P}^4,d)\, ,$$
has virtual dimension 0.
The Gromov-Witten invariants,
\begin{equation}\label{fredfredgw}
N^{\mathsf{GW}}_{g,d}\, =\,
\langle 1\rangle_{g,d}^{\mathsf{GW}}\,=\, \int_{[\overline{M}_g(X_5,d)]^{vir}} 1\, ,
\end{equation}
have been studied for more than 20 years, 
see \cite{CKatz,FP,Kon} for an
introduction to the subject.

The theory of stable quotients developed in \cite{MOP}
was partially inspired by the question of finding a geometric approach
to a higher genus linear sigma model. 
The moduli space
of stable quotients for the quintic,
$$\overline{Q}_g(X_5,d)\subset \overline{Q}_g(\mathbb{P}^4,d)\, ,$$
was defined in \cite[Section 9]{MOP}.
The existence of a natural obstruction theory on
$\overline{Q}_g(X_5,d)$
and a virtual
fundamental class ${[\overline{Q}_g(X_5,d)]^{vir}}$ is easily seen{\footnote{For stability,
marked points are required in genus 0 and
positive degree is required in genus 1.}}
in genus 0 and 1.
A proposal in higher genus for the obstruction theory and
virtual class was made in \cite{MOP} and was carried out in
significantly greater generality in the setting of quasimaps in
\cite{CKM}. The associated integral theory is defined by
\begin{equation}
\label{fredfred}
N^{\mathsf{SQ}}_{g,d}
\,=\, \langle 1 \rangle_{g,d}^{\mathsf{SQ}}
\, =\, \int_{
[\overline{Q}_g(X_5,d)]^{vir}} 1\, .
\end{equation}
In genus 0 and 1, the invariants \eqref{fredfred}
 were calculated 
in \cite{CZ} and \cite{KL} respectively.
The answers on the stable quotient side {\em exactly}
match the string theoretic B-model for the quintic in
genus 0 and 1.

A relationship in every genus between the
Gromov-Witten and stable quotient invariants of
the quintic has been proven by Ciocan-Fontanine
and Kim \cite{CKw}.{\footnote{A second proof (in most cases) can
be found in  \cite{CJR}.}}
Let $H\in H^2(X_5,\mathbb{Z})$ be the hyperplane class of
the quintic,
and let
$$\mathcal{F}_{g,n}^{\mathsf{GW}}(Q)\, =\, \langle \, \underbrace{H,\ldots,H}_{n}  \, \rangle_{g,n}^{\mathsf{GW}} \, =\, 
\sum_{d=0}^\infty  Q^d
\int_{[\overline{M}_{g,n}(X_5,d)]^{vir}} \prod_{i=1}^n
\text{ev}_i^*(H)\,  ,$$
$$\mathcal{F}_{g,n}^{\mathsf{SQ}}(q) \, =\, \langle \, \underbrace{H,\ldots,H}_{n}  \, \rangle_{g,n}^{\mathsf{SQ}} \, =\, 
\sum_{d=0}^\infty  q^d
\int_{[\overline{Q}_{g,n}(X_5,d)]^{vir}} \prod_{i=1}^n
\text{ev}_i^*(H)\, $$
be the Gromov-Witten and stable quotient series
respectively (involving the pointed moduli spaces and the
evaluation morphisms at the markings).
Let 
$$I_0(q)=\sum_{d=0}^\infty q^d \frac{(5d)!}{(d!)^5}\, ,  \ \ \ 
I_1(q)= \log(q)I_0(q) + 5 \sum_{d=1}^\infty q^d \frac{(5d)!}{(d!)^5} 
\left( \sum_{r=d+1}^{5d} \frac{1}{r}\right)\, .
$$
The mirror map is defined by
$$Q(q) = \exp\left(\frac{I_1(q)}{I_0(q)}\right) =
q\cdot \exp\left( \frac{5 \sum_{d=1}^\infty q^d \frac{(5d)!}{(d!)^5} 
\left( \sum_{r=d+1}^{5d} \frac{1}{r}\right)}{\sum_{d=0}^\infty q^d \frac{(5d)!}{(d!)^5}} \right)
\, .$$
The relationship  between the Gromov-Witten and stable quotient
invariants of the quintic in case $$2g-2+n>0$$ is given by the following result
\cite{CKw}:
\begin{equation}\label{345}
\mathcal{F}_{g,n}^{\mathsf{GW}}(Q(q)) =
I_0(q)^{2g-2+n} \cdot \mathcal{F}_{g,n}^{\mathsf{SQ}}(q)\, .
\end{equation}
The transformation \eqref{345} shows the stable quotient
theory matches the string theoretic
B-model series for the quintic $X_5$.


\subsection{Formal quintic invariants}\label{ham}
The (conjectural) holomorphic anomaly equation is a beautiful property
of the string theoretic B-model series which has been used
effectively since \cite{BCOV}. Since the stable quotients
invariants provide a geometric proposal for the B-model series,
we should look for the geometry of the holomorphic
anomaly equation in the moduli space of stable quotients.

 A particular twisted theory
 on $\PP^4$ is
 related to the quintic
 3-fold. Let the algebraic torus
$$\mathsf{T}=(\com^*)^5$$
 act with the standard linearization
 on $\PP^4$ with weights
 $\lambda_0,\ldots,\lambda_4$
 on the vector space  $H^0(\PP^4,\mathcal{O}_{\PP^4}(1))$.
Let
\begin{equation}\label{gred}
\mathsf{C} \rightarrow \overline{M}_g(\mathbb{P}^4,d)
\, , \ \ \ 
f:\mathsf{C}\rightarrow \PP^4
\, , \ \ \ 
\mathsf{S}=f^*\mathcal{O}_{\PP^4}(-1) \rightarrow \mathsf{C}\, 
\end{equation}
be the universal curve,
the universal map,
and the universal bundle over
the moduli space of stable maps
--- all
equipped with canonical $\T$-actions.
We define the {\em formal quintic invariants} following
\cite{LP1} by{\footnote{The negative
exponent denotes the dual:
$\mathsf{S}$ is a line bundle and $\mathsf{S}^{-5}=(\mathsf{S}^\star)^{\otimes 5}$.}}
\begin{equation}\label{fredfredfred}
\widetilde{N}_{g,d}^{\mathsf{GW}}= \int_{[\overline{M}_g(\mathbb{P}^4,d)]^{vir}} e(R\pi_*(\mathsf{S}^{-5}))\, , \   
\end{equation}
where $e(R\pi_*(\mathsf{S}^{-5}))$
is the equivariant Euler class defined
{\em after} localization. More precisely, on each $\T$-fixed locus of
$\overline{M}_g(\mathbb{P}^4,d)$, both
$$R^0\pi_*(\mathsf{S}^{-5}) \ \ \ \text{and} \ \ \
R^1\pi_*(\mathsf{S}^{-5})$$
are vector bundles with moving weights, so
$$e(R\pi_*(\mathsf{S}^{-5})) = \frac{c_{\text{top}}(R^0\pi_*(\mathsf{S}^{-5}))}
{c_{\text{top}}(R^1\pi_*(\mathsf{S}^{-5}))}$$
is well-defined.
The integral \eqref{fredfredfred} is
homogeneous of degree 0 in localized
equivariant cohomology,
$$\int_{[\overline{M}_g(\mathbb{P}^4,d)]^{vir}} e(R\pi_*(\mathsf{S}^{- 5}))
\, \in \,{\mathbb{Q}}(\lambda_0,\ldots,\lambda_4)\, ,$$
and defines a rational number 
 $\widetilde{N}_{g,d}^{\mathsf{GW}}\in {\mathbb{Q}}$
after the
specialization{\footnote{Since the formal quintic theory is
homogeneous of degreee 0, the specialization $\lambda_i = \zeta^i\lambda_0$ could
also be taken as in \cite{LP1}. However, for other specializations,
we expect the finite generation of Theorem \ref{ooo5} and the holomorphic
anomaly equation of Theorem \ref{ttt5} to take different forms. A study of the
dependence on specialization will appear
in \cite{Lho999}. 
For local $\mathbb{P}^2$
considered in \cite{LP1} and $\mathbb{C}^3/\mathbb{Z}_3$ considered in \cite{LPC3Z3},
the theories are independent of specialization.}}
$$\lambda_i = \zeta^i$$
for a primitive fifth root of unity
$\zeta^5=1$.

 
Our main result here is  that the   holomorphic anomaly equations
conjectured for the
true quintic theory \eqref{fredfredgw}
are satisfied by the formal 
quintic theory
\eqref{fredfredfred}.
%
In particular,
the formal quintic theory  and the true quintic theory
should be related by transformations which respect the
holomorphic anomaly equations. In a recent breakthrough
by Q. Chen, S. Guo, F. Janda, and Y. Ruan, precisely
such a transformation was found via the virtual geometry of
new moduli spaces intertwining the formal and true theories
of the quintic.

\subsection{Holomorphic anomaly for the
formal quintic}

We state here the precise form of the holomorphic anomaly equations for the formal quintic.

Let $H\in H^2(\PP^4,\ZZ)$ be the hyperplane class on $\PP^4$, and let
\begin{eqnarray*}
\widetilde{\mathcal{F}}_{g}^{\mathsf{GW}}(Q)&=& \sum_{d=0}^{\infty}Q^d \int_{[\overline{M}_g(\PP^4,d)]^{vir}}e(R\pi_*(\mathsf{S}^{-5}))\,, \\
\widetilde{\mathcal{F}}_{g}^{\mathsf{SQ}}(Q)&=&\sum_{d=0}^{\infty}Q^d \int_{[\overline{Q}_g(\PP^4,d)]^{vir}}e(R\pi_*(\mathsf{S}^{-5}))\,
\end{eqnarray*}
be the formal Gromov-Witten and formal stable quotient series respectively (involving the evaluation morphisms at the markings). The relationship between the formal Gromov-Witten and formal stable quotient invariants of quintic in case of $2g-2+n > 0$ follows from \cite{CKg}:
\begin{equation}\label{jjed}
\widetilde{\mathcal{F}}_{g}^{\mathsf{GW}}(Q(q)) =
I_0(q)^{2g-2} \cdot \widetilde{\mathcal{F}}_{g}^{\mathsf{SQ}}(q)
\end{equation}
with respect to the true quintic mirror map
$$Q(q) = \exp\left(\frac{I_1(q)}
{I_0(q)}\right)\, = \, 
q\cdot \exp\left( \frac{5 \sum_{d=1}^\infty q^d \frac{(5d)!}{(d!)^5} 
\left( \sum_{r=d+1}^{5d} \frac{1}{r}\right)}{\sum_{d=0}^\infty q^d\frac{(5d)!}{(d!)^5}} \right)
\, .$$

In order to state the holomorphic anomaly equations,
we require several series in $q$. First, let
$$L(q) \, =\,   (1-5^5 q)^{-\frac{1}{5}}\,  = \, 1+625q+117185 q^2 +\ldots\, .$$
Let $\mathsf{D}=q\frac{d}{dq}$, and let
$$C_0(q)= I_0\, , \ \ \ 
C_1(q)= \mathsf{D} \left( \frac{I_1}
{I_0}\right)\, ,$$
where $I_0$ and $I_1$ are the hypergeometric series
appearing in the mirror map for the true quintic theory.
We define
\begin{eqnarray*}
K_2(q)&=& -\frac{1}{L^5} \frac{\DD C_0}{C_0}\,  ,
\\
A_2(q)&=& \frac{1}{L^5}\left( -\frac{1}{5}\frac{\DD C_1}{C_1}-\frac{2}{5}\frac{\DD C_0}{C_0}-\frac{3}{25}\right)\, ,\\
A_4(q) &=& \frac{1}{L^{10}} \Bigg(-\frac{1}{25}\left(
\frac{\DD C_0}{C_0}\right)^2-\frac{1}{25}
\left(\frac{\DD C_0}{C_0}\right)\left(\frac{\DD C_1}{C_1}\right)
\, \\
& &
+\frac{1}{25}\DD\left(\frac{\DD C_0}{C_0}\right)+\frac{2}{25^2} \Bigg)\,  ,\\
A_6(q) &=&\frac{1}{31250L^{15} }\Bigg( 4+125 \DD\left(\frac{\DD C_0}{C_0}\right)
+50\left(\frac{\DD C_0}{C_0}\right)
\left(1+10 \DD \left(\frac{\DD C_0}{C_0}
\right)\right)\,    \\
& & -5L^5\Bigg(1+10\left(\frac{\DD C_0}{C_0}\right)+25
\left(\frac{\DD C_0}{C_0}\right)^2+25\DD
\left(\frac{q\frac{d}{dq}C_0}{C_0}\right)\Bigg)\,  \\
& &+125\DD^2\left(\frac{\DD C_0}{C_0}\right)-125\left(\frac{\DD C_0}{C_0}\right)^2\Bigg(\left(\frac{\DD C_1}{C_1}\right)-1\Bigg) \Bigg)\, .
\end{eqnarray*}
Let $T$ be the standard coordinate mirror to $t=\log(q)$,
\begin{align*} T= \frac{I_1(q)}{I_0(q)}\, .
\end{align*}
Then $Q(q)=\exp(T)$ is the mirror map.

Define a new series 
$$\widetilde{\mathcal{F}}_{g}^{\mathsf{B}}=I_0^{2g-2}\cdot\widetilde{\mathcal{F}}^{\mathsf{SQ}}_g$$
motivated by \eqref{jjed}. The superscript $\mathsf{B}$ here is for the $B$-model.
Let $$\mathbb{C}[L^{\pm1}][A_2,A_4,A_6,C_0^{\pm 1},C_1^{-1},K_2]$$ 
be the free
polynomial ring over $\mathbb{C}[L^{\pm1}]$.

\begin{Thm} \label{ooo5} For the series 
$\widetilde{\mathcal{F}}_g^{\mathsf{B}}$
associated to the formal quintic, 
\begin{enumerate}
\item[(i)]
$\widetilde{\mathcal{F}}_g^{\mathsf{B}} (q) \in \mathbb{C}[L^{\pm1}][A_2,A_4,A_6,C_0^{\pm 1},C_1^{-1},K_2]$
for $g\geq 2$, \vspace{5pt}

\item[(ii)]
$\frac{\partial^k \widetilde{\mathcal{F}}_g^{\mathsf{B}}}{\partial T^k}(q) \in \mathbb{C}[L^{\pm1}][A_2,A_4,A_6,C_0^{\pm 1},C_1^{-1},K_2]$ for  $g\geq 1$, $k\geq 1$,
\vspace{5pt}

\item[(iii)]
${\frac{\partial^k \widetilde{\mathcal{F}}_g^{\mathsf{B}}}{\partial T^k}}$ is homogeneous 
with respect  to $C_1^{-1}$ 
of degree $k$.
\end{enumerate}
\end{Thm}


We follow here the {\em canonical lift} convention of \cite[Section 0.4]{LP1}.
When we write
$$\widetilde{\mathcal{F}}_g^{\mathsf{B}} (q) \in \mathbb{C}[L^{\pm1}][A_2,A_4,A_6,C_0^{\pm 1},C_1^{-1},K_2]\, , $$
we mean that the series $\widetilde{\mathcal{F}}_g^{\mathsf{B}} (q)$
has a canonical lift to the
free algebra. The question of
uniqueness of the lift has to do
with the algebraic independence of
the series 
$$A_2(q)\, , \  A_4(q)\,, \   A_6(q)\, ,\  
C_0^\pm(q)\,,  \ C_1^{-1}(q)\,,  \ K_2(q)$$
which we do not address nor require.

\begin{Thm} \label{ttt5} The holomorphic anomaly equations
for the series $\widetilde{\mathcal{F}}^{\mathsf{B}}_g$
associated to the formal quintic hold for $g\geq 2$:

\begin{multline*}
\frac{1}{C_0^2C_1^2}\frac{\partial \widetilde{\mathcal{F}}_g^{\mathsf{B}}}{\partial{A_2}}-\frac{1}{5C_0^2C_1^2}\frac{\partial \widetilde{\mathcal{F}}_g^{\mathsf{B}}}{\partial{A_4}}K_2+\frac{1}{50C_0^2C_1^2}\frac{\partial \widetilde{\mathcal{F}}_g^{\mathsf{B}}}{\partial{A_6}}K_2^2
=
\frac{1}{2}\sum_{i=1}^{g-1} 
\frac{\partial \widetilde{\mathcal{F}}_{g-i}^{\mathsf{B}}}{\partial{T}}
\frac{\partial \widetilde{\mathcal{F}}_i^{\mathsf{B}}}{\partial{T}}
+
\frac{1}{2}\,
\frac{\partial^2 \widetilde{\mathcal{F}}_{g-1}^{\mathsf{B}}}{\partial{T}^2}\,
,
\end{multline*}

\begin{align*}
\frac{\partial \widetilde{\mathcal{F}}_g^{\mathsf{B}}}{\partial K_2}=0\,.
\end{align*}
\end{Thm}

\noindent 
The equality of Theorem \ref{ttt5} holds in 
the ring
$$\mathbb{C}[L^{\pm1}][A_2,A_4,A_6,C_0^{\pm1},C_1^{-1},K_2]\, .$$
Theorem \ref{ttt5} exactly matches{\footnote{Our functions
$K_2$ and $A_{2k}$ 
 are normalized differently with respect to $C_0$ and $C_1$.
The dictionary to exactly match the notation of \cite[(2.52)]{ASYZ} is to 
multiply our $K_2$ by $(C_0C_1)^2$ and our $A_{2k}$ by $(C_0C_1)^{2k}$.}}
the
conjectural holomorphic anomaly equations
 \cite[(2.52)]{ASYZ} for
the true quintic theory $I_0^{2g-2}\cdot\mathcal{F}_g^{\mathsf{SQ}}$.

The first holomorphic anomaly equation of
Theorem \ref{ttt5} was announced in our 
paper \cite{LP1} in February 2017
where a parallel study of the toric Calabi-Yau $K\proj^2$ was developed.
In January 2018 at the {\em Workshop on higher genus} at ETH Z\"urich,
Shuai Guo of Peking University informed us that our same argument also yields the
second holomorphic anomaly equation 
\begin{equation}\label{gg559}
\frac{\partial \widetilde{\mathcal{F}}_g^{\mathsf{B}}}{\partial K_2}=0\,.
\end{equation}
In fact, we had incorrectly thought \eqref{gg559}
would fail in the formal theory and would {\em not} have included \eqref{gg559}
without communication with Guo, so Guo should
be credited with the first proof of \eqref{gg559}.

\subsection{Constants}
Theorem \ref{ttt5} and
a few observations
determine   $\widetilde{\mathcal{F}}_g^{\mathsf{B}}$
from the lower genus data
$$\Big\{\, h<g \,\Big| \, \widetilde{\mathcal{F}}_h^{\mathsf{B}}\, \Big\}$$
and finitely many constants of integration.
The additional 
observations{\footnote{See Remark \ref{C0C1}.}}
required are:
\begin{enumerate}
    \item [(i)]
    The proof of part (i) of
Theorem \ref{ooo5} shows that 
$$\widetilde{\mathcal{F}}^{\mathsf{B}}_g \in \mathbb{C}[L^{\pm1}][A_2,A_4,A_6,C_0^{\pm1},C_1^{-1},K_2]\,$$ 
does {\em not} depend on $C_1^{-1}$. Hence, every term (both on the left and
right) 
in the first holomorphic 
anomaly equation of Theorem \ref{ttt5} is of degree 2 in $C_1^{-1}$. After multiplying by $C_1^2$,
no $C_1$ dependence remains.
\item[(ii)]
 The proof of Theorem \ref{ooo5} shows that all terms
 in the first equation are homogeneous of degree $2g-4$ with respect to $C_0$.
 After dividing by $C_0^{2g-4}$,
 no $C_0$ dependence remains.
 \end{enumerate}
 Therefore, the first holomorphic anomaly equation of Theorem \ref{ttt5} may be viewed as holding in $\CC[L^{\pm 1}][A_2,A_4,A_6,K_2]$. 
 
Since the second holomorphic anomaly equation \eqref{gg559} implies
$\widetilde{\mathcal{F}}_g^{\mathsf{B}}$ has no $K_2$ dependence,
the first holomorphic anomaly equation determines the each of the three derivatives
$$\frac{\partial \widetilde{\mathcal{F}}_g^{\mathsf{B}}}{\partial{A_2}}\, , \ \ \  \frac{\partial \widetilde{\mathcal{F}}_g^{\mathsf{B}}}{\partial{A_4}}\, ,\ \ \ \frac{\partial \widetilde{\mathcal{F}}_g^{\mathsf{B}}}{\partial{A_6}}\, .$$
Hence, Theorem \ref{ttt5} determines 
$$C_0^{2-2g}\cdot\widetilde{\mathcal{F}}^{\mathsf{B}}_g$$ uniquely as a polynomial in $A_2,A_4,A_6$ up to a constant term in $\CC[L^{\pm 1}]$. In fact, the degree of the constant term can be bounded (as will be seen in Section \ref{Btd}). So Theorem \ref{ttt5} determines $\widetilde{\mathcal{F}}^{\mathsf{B}}_g$ from lower genus data together with finitely many constants of integration.

The constants of integration for the formal quintic can be effectively
computed via the localization formula, but whether there exists a
closed formula determining the constants
is an interesting open question.

\subsection{Acknowledgments} 
We thank I.~Ciocan-Fontanine, E. Clader, Y. Cooper, B.~Kim, 
A.~Klemm, Y.-P. Lee, A.~Marian, M.~Mari\~no, D.~Maulik, D.~Oprea,  E.~Scheidegger, Y. Toda, and A.~Zinger
for discussions over the years 
about the moduli space of stable quotients and the invariants of Calabi-Yau geometries. The work of Q. Chen, F. Janda,
S. Guo, and Y. Ruan as presented at the 
{\em Workshop on higher genus}  is crucial for the
wider study of the formal (and true) quintic. We are very grateful
to them for sharing their ideas with us.

R.P. was partially supported by 
SNF-200020182181, ERC-2012-AdG-320368-MCSK, 
ERC-2017-AdG-786580-MACI,
SwissMAP, and
the Einstein Stiftung. 
H.L. was supported by the grant ERC-2012-AdG-320368-MCSK and
ERC-2017-AdG-786580-MACI.

This project has received funding from the European Research Council (ERC) under the European Union’s
Horizon 2020 research and innovation program (grant agreement No 786580).

\section{Localization graphs}

\label{locq}
\subsection{Torus action}
Let $\mathsf{T}=(\com^*)^{m+1}$ act diagonally on the vector space $\mathbb{C}^{m+1}$
with weights
$$-\lambda_0, \ldots, -\lambda_m\, .$$
Denote the $\mathsf{T}$-fixed points of 
the induced $\mathsf{T}$-action on $\proj^m$ by
$$p_0, \ldots, p_{m}\, . $$
The weights of $\mathsf{T}$ on the tangent space $T_{p_j}(\proj^m)$ are
$$\lambda_j-\lambda_0, \ldots, \widehat{\lambda_j-\lambda_j}  ,\ldots, \lambda_j-\lambda_{m}\, .$$

There is an induced $\mathsf{T}$-action on 
the moduli space 
$\overline{Q}_{g,n}(\proj^m,d)$.
The localization formula of \cite{GP} applied to the  virtual fundamental class 
$[\overline{Q}_{g,n}(\proj^m,d)]^{vir}$ will play a fundamental role our paper.
The $\mathsf{T}$-fixed loci are represented in terms of dual graphs,
and the contributions of the $\mathsf{T}$-fixed loci are given by
tautological classes. The formulas here are
standard. We precisely follow the
notation of \cite[Section 2]{LP1}.

\subsection{Graphs}\label{grgr}
Let the genus $g$ and the number of markings $n$ for the moduli
space be in
the stable range
\begin{equation}\label{dmdm}
2g-2+n>0\, .
\end{equation}
We can organize the $\mathsf{T}$-fixed loci 
of $\overline{Q}_{g,n}(\proj^m,d)$
according to decorated graphs.
A {\em decorated graph} $\Gamma \in \mathsf{G}_{g,n}(\proj^m)$ consists 
of the data $(\mathsf{V}, \mathsf{E}, 
\mathsf{N}, \mathsf{g}, \mathsf{p} )$ where
\begin{enumerate}
 \item[(i)] $\mathsf{V}$ is the vertex set, 
 \item[(ii)] $\mathsf{E}$ is the edge set (including
 possible self-edges),
 \item[(iii)] $\mathsf{N} : \{1,2,..., n\} \rightarrow \mathsf{V}$ is the marking
 assignment,
   \item[(iv)] $\mathsf{g}: \mathsf{V} \rightarrow \ZZ_{\geq 0}$ is a genus
 assignment satisfying
 $$g=\sum_{v \in V} \mathsf{g}(v)+h^1(\Gamma)\, $$
and for which $(\mathsf{V},\mathsf{E},\mathsf{N},\mathsf{g})$ is stable graph{\footnote
{Corresponding to a stratum of the moduli space
of stable curves $\overline{M}_{g,n}$.}}, 
 \item[(v)] $\mathsf{p} : \mathsf{V} \rightarrow ({\PP ^m})^{\mathsf{T}}$ is an assignment of a $\mathsf{T}$-fixed point $\mathsf{p} (v)$ to each vertex $v \in \mathsf{V}$.
\end{enumerate}
The markings $\mathsf{L}=\{1,\ldots,n\}$ are often called {\em legs}.

To each decorated graph $\Gamma\in \mathsf{G}_{g,n}(\proj^m)$, we associate the set of ${\mathsf{T}}$-fixed loci of  
$$\sum_{d\geq 0} \left[\overline{Q}_{g, n} (\proj^m, d)\right]^{\vir} q^d$$
with elements described as follows:
\begin{enumerate}
 
 \item[(a)] If $\{v_{i_1},\ldots,v_{i_k}\}=\{v\, |\, \mathsf{p}(v)=p_i\}$, then $f^{-1}(p_i)$ is a disjoint union of connected stable curves of genera $\mathsf{g}(v_{i_1}),\ldots, \mathsf{g}(v_{i_k})$ and finitely many points.
 
 
 \item[(b)] There is a bijective
  correspondence between the connected components of $C \setminus D$ and the set of edges{\footnote{Self-edges correspond to loops of ${\mathsf{T}}$-invariant
  rational curves.}} and legs of $\Gamma$ respecting  vertex incidence where $C$ is domain curve and $D$ is union of all subcurves of $C$ which appear in (a). 
  
\end{enumerate}
We write the localization formula as
$$\sum_{d\geq 0} \left[\overline{Q}_{g, n} (\PP^m, d)\right]^{\vir} q^d =
\sum_{\Gamma\in \mathsf{G}_{g,n}(\PP^m)} \text{Cont}_\Gamma\, .$$
While $\mathsf{G}_{g,n}(\proj^m)$ is a finite set,
each contribution $\text{Cont}_\Gamma$ is
a series in $q$ obtained from
an infinite sum over all edge possibilities (b).

\subsection{Unstable graphs}
The moduli spaces of stable quotients $$\overline{Q}_{0,2}(\proj^m,d) \ \ \
\text{and} \ \ \ \overline{Q}_{1,0}(\proj^m,d)$$
for $d>0$
are the only{\footnote{The moduli spaces
$\overline{Q}_{0,0}(\proj^m,d)$ and
$\overline{Q}_{0,1}(\proj^m,d)$
are empty by the definition of a stable
quotient.}}
cases where
the pair $(g,n)$ does 
{\em not} satisfy the Deligne-Mumford stability condition 
\eqref{dmdm}. 

An appropriate set of decorated graphs $\mathsf{G}_{0,2}(\PP^m)$
 is easily defined: The graphs
$\Gamma \in \mathsf{G}_{0,2}(\PP^m)$ all have 2 vertices
connected by a single edge. Each vertex carries a marking.
All  of the  conditions (i)-(v)
of Section \ref{grgr} are satisfied
except for the stability of $(\mathsf{V},\mathsf{E}, \mathsf{N},\gamma)$.
The  localization formula holds,
\begin{eqnarray}\label{ddgg}
\sum_{d\geq 1} \left[\overline{Q}_{0, 2} (\PP^m, d)\right]^{\vir} q^d &=&
\sum_{\Gamma\in \mathsf{G}_{0,2}(\proj^m)} \text{Cont}_\Gamma\,,
\end{eqnarray}
For $\overline{Q}_{1,0}(\proj^m,d)$, the matter
is more problematic --- usually a marking
is introduced to break the symmetry.

\section{Basic correlators}\label{bcbcbc}
\subsection{Overview}
We review here basic generating series in $q$ which 
arise in  the genus 0 theory of quasimap invariants. The series
will play a fundamental role in
the calculations of Sections \ref{hgi} - \ref{hafp}
related
to the holomorphic anomaly equation for formal quintic invariants.

We fix a torus action $\mathsf{T}=(\CC^*)^5$ on $\PP^4$ with
weights{\footnote{The associated weights on
$H^0(\PP^4,\mathcal{O}_{\PP^4}(1))$ are
$\lambda_0,\lambda_1,\lambda_2,\lambda_3,\lambda_4$
and so match the conventions of
Section \ref{ham}.}}
$$-\lambda_0, -\lambda_1, -\lambda_2, -\lambda_3, -\lambda_4$$
on the vector space $\mathbb{C}^5$.
The $\T$-weight on the fiber over
$p_i$ of the canonical
bundle 
\begin{equation}\label{pqq9}
\mathcal{O}_{\PP^4}(5) \rightarrow \PP^4
\end{equation}
is $5\lambda_i$.

For our formal quintic theory, we will use the specialization
\begin{equation}
\label{spez}
\lambda_i=\zeta^i
\end{equation}
where $\zeta$ is the primitive fifth root of unity.
Of course, we then have
\begin{eqnarray*}
    \lambda_0+\lambda_1+\lambda_2+\lambda_3+\lambda_4&=&0\, , \\
    \sum_{i \ne j}\lambda_i \lambda_j&=&0\, , \\
    \sum_{i \ne j \ne k} \lambda_i \lambda_j \lambda_k&=&0\,  , \\
    \sum_{i \ne j \ne k \ne l} \lambda_i \lambda_j \lambda_k \lambda_l&=&0\,.
\end{eqnarray*}

\subsection{First correlators}
We will require several correlators defined 
via  the Euler class{\footnote{Equation
\eqref{34gg34} is the definition
of $e(\text{Obs})$. The
right side of \eqref{34gg34}
is defined after localization as
explained in Section \ref{ham}.},
\begin{equation}\label{34gg34}
e(\text{Obs})= e(R\pi_* (\mathsf{S}^{-5}))\, ,
\end{equation}
associated to the formal quintic geometry
on the moduli space $\overline{Q}_{g,n}(\PP^4,d)$.
The first two are obtained from
standard stable quotient invariants.
For $\gamma_i \in H^*_{\T} (\PP^4)$, let
\begin{eqnarray*}
    \Big \langle \gamma_1\psi^{a_1}, ...,\gamma_n\psi^{a_n} \Big\rangle_{g,n,d}^{\mathsf{SQ}}&=& 
    \int_{[\overline{Q}_{g,n}(\PP^4,d)]^{\vir}} 
    e(\text{Obs})\cdot
    \prod_{i=1}^n \text{ev}_i^*(\gamma_i)\psi_i^{a_i},\\
      \Big\langle \Big\langle \gamma _1\psi  ^{a_1} , ..., \gamma _n\psi  ^{a_n} \Big\rangle \Big\rangle _{0, n}^{\mathsf{SQ}} 
    &=& \sum _{d\geq 0}\sum_{k\geq 0} \frac{q^{d}}{k!}
 \Big\lan    \gamma _1\psi  ^{a_1} , ..., \gamma _n\psi  ^{a_n} , t, ..., t  \Big\ran_{0, n+k, d}^{\mathsf{SQ}} , 
 \end{eqnarray*}
 where, in the second series,
 $t \in H_{\T}^* (\PP^4)$.
 We will systematically use the quasimap notation $0+$
for stable quotients,
\begin{eqnarray*}
    \Big \langle \gamma_1\psi^{a_1}, ...,\gamma_n\psi^{a_n} \Big\rangle_{g,n,d}^{0+}&=&
    \Big \langle \gamma_1\psi^{a_1}, ...,\gamma_n\psi^{a_n} \Big\rangle_{g,n,d}^{\mathsf{SQ}} \\
      \Big\langle \Big\langle \gamma _1\psi  ^{a_1} , ..., \gamma _n\psi  ^{a_n} \Big\rangle \Big\rangle _{0, n}^{0+} 
    &=& 
    \Big\langle \Big\langle \gamma _1\psi  ^{a_1} , ..., \gamma _n\psi  ^{a_n} \Big\rangle \Big\rangle _{0, n}^{\mathsf{SQ}}\, . 
 \end{eqnarray*}

\subsection{Light markings}\label{lightm}
Moduli of quasimaps can be considered with $n$ ordinary (weight 1) markings and $k$ light 
(weight $\epsilon$) markings{\footnote{See Sections 2 and 5 of \cite{BigI}.}},
$$\overline{Q}^{0+,0+}_{g,n|k}(\PP^4,d)\, .$$
Let $\gamma_i \in H^*_{\T} (\PP^4)$ be
equivariant cohomology classes, and
let
$$\delta _j \in H^*_{\T} ([\mathbb{C}^5/\com^* ])$$ 
be classes on the stack quotient. 
Following the notation of \cite{KL}, 
we define series for the formal quintic geometry,

\begin{multline*}
    \Big\lan \gamma _1\psi  ^{a_1} , \ldots, \gamma _n\psi  ^{a_n} ;  \delta _1, \ldots, \delta _k \Big\ran _{g, n|k, d}^{0+, 0+}  = \\
\int _{[\overline{Q}^{0+, 0+}_{g, n|k} (\PP^4, d)]^{\vir}} 
e(\text{Obs})\cdot
\prod _{i=1}^n \text{ev}_i^*(\gamma _i)\psi _i ^{a_i} 
\cdot \prod _{j=1}^k \widehat{\text{ev}}_j ^* (\delta _j)\, , 
\end{multline*}
\begin{multline*}
\Big \langle \Big\langle \gamma _1\psi  ^{a_1} , \ldots, \gamma _n\psi  ^{a_n} \Big\rangle\Big\rangle _{0, n}^{0+, 0+} 
= \\ \sum _{d\geq 0}\sum_{k\geq 0} \frac{q^{d}}{k!}
 \Big\lan    \gamma _1\psi  ^{a_1} , \ldots, \gamma _n\psi  ^{a_n} ; {t}, \ldots, {t}  
 \Big\ran_{0, n|k, d}^{0+, 0+} \, ,
 \end{multline*}
 where, in the second series,
 ${t} \in H_{\T}^* ([\mathbb{C}^5/\com^* ])$.
 
 For each $\T$-fixed point $p_i\in \PP^4$, let 
 $$e_i= \frac{e(T_{p_i}(\PP^4))}{5\lambda_i}$$
 be the equivariant Euler class of
 the tangent space of $\PP^4$ at $p_i$ with twist by $\mathcal{O}_{\PP^4}(5)$. Let
 $$ \phi_i=\frac{\prod_{j \ne i}(H-\lambda_j)}{5\lambda_i e_i}, \ \ \phi^i=e_i \phi_i\ \ \in H^*_{\T}(\PP^4)\, $$ be cycle classes. Crucial for us are the series
 \begin{align*}
\mathds{S}_i(\gamma) & = e_i \Big\langle \Big\langle  \frac{\phi _i}{z-\psi} , \gamma 
\Big\rangle \Big\rangle _{0, 2}^{0+,0+}\ \ ,  \\
\mathds{V}_{ij}  & =  
\Big\langle \Big\langle  \frac{\phi _i}{x- \psi } ,  \frac{\phi _j}{y - \psi } 
\Big\rangle \Big\rangle  _{0, 2}^{0+,0+}  \ \ . 
 \end{align*}
Unstable degree 0 terms are included by hand in the
above formulas. For $\mathds{S}_i(\gamma)$, the unstable degree 0 term is
$\gamma|_{p_i}$. For $\mathds{V}_{ij}$, the unstable degree 0 term is
$\frac{\delta_{ij}}{e_i(x+y)}$.

 We also write $$\mathds{S}(\gamma)=\sum_{i=0}^4 {\phi_i} \mathds{S}_i(\gamma)\, .$$ 
The series $\mathds{S}_i$ and $\mathds{V}_{ij}$
 satisfy the basic relation
\begin{equation}  \label{wdvv} e_i\mathds{V}_{ij} (x, y)e_j   = 
\frac{\sum _{k=0}^4 \mathds{S}_i (\phi_k)|_{z=x} \, \mathds{S}_j(\phi ^k )|_{z=y}}{x+ y}\,   \end{equation}
 proven{\footnote{In Gromov-Witten
 theory, a parallel relation is
 obtained immediately from the
 WDDV equation and the string equation.
 Since the map forgetting a point
 is not always well-defined for
 quasimaps, a different argument 
 is needed here \cite{CKg}}} in \cite{CKg}.
 
 Associated to each $\T$-fixed point $p_i\in \PP^4$,
 there is a special $\T$-fixed point locus, 
\begin{equation}\label{ppqq}
\overline{Q}^{0+, 0+}_{0, k|m} (\PP^4,d) ^{\T, p_i} \subset
\overline{Q}^{0+, 0+}_{0, k|m}(\PP^4, d)\, ,
\end{equation}
where all markings lie on a single connected
genus 0 domain component contracted to $p_i$.
Let $\text{Nor}$ denote the equivariant
normal bundle 
of $Q^{0+, 0+}_{0, n|k} (\PP^4,d) ^{\T, p_i}$
with respect to the embedding \eqref{ppqq}.
Define 
\begin{multline*}
\Big\lan \gamma _1\psi  ^{a_1} , \ldots, \gamma _n\psi  ^{a_n} ;  \delta _1, ..., \delta _k \Big\ran _{0, n|k, d}^{0+, 0+, p_i}  
=\\
\int _{[\overline{Q}^{0+, 0+}_{0, n|k} (\PP^4, d) ^{\T, p_i}]} 
\frac{e(\text{Obs})}{e(\text{Nor})}\cdot
\prod _{i=1}^n \text{ev}_i^*(\gamma _i)\psi _i ^{a_i} \cdot
\prod _{j=1}^k \widehat{\text{ev}}_j ^* (\delta _j) \, ,
\end{multline*}

\begin{multline*}
  \Big\langle \Big\langle
  \gamma _1\psi  ^{a_1} ,\ldots, \gamma _n\psi  ^{a_n} 
  \Big\rangle \Big\rangle  _{0, n}^{0+, 0+, p_i} =\\
 \sum _{d\geq 0}\sum_{k\geq 0} \frac{q^{d}}{k!}
 \Big\lan    \gamma _1\psi  ^{a_1} , \ldots, \gamma _n\psi  ^{a_n} ; {t}, \ldots, {t}  \Big \ran_{0, n|k, \beta}^{0+, 0+, p_i} \, .
\end{multline*}

\subsection{Graph spaces and I-functions}
\subsubsection{Graph spaces}
The {big I-function} is defined in \cite{BigI}
via the geometry of weighted quasimap graph spaces. 
We briefly summarize the constructions of \cite{BigI}
in the special case of 
%
%
$(0+,0+)$-stability. The more general weightings
discussed in \cite{BigI} will not be
needed here.

As in Section \ref{lightm}, we consider the quotient
$$\com^5/\com^*$$
associated to $\PP^4$.
Following \cite{BigI},
there is a $(0+,0+)$-{\em stable quasimap graph space}
 \begin{equation}\label{xmmx}
     \mathsf{QG}_{g, n|k, d }^{0+,0+} ([\com^5/\com^*] ) \, .
 \end{equation}
A $\CC$-point of the graph space is described by data 
$$((C, {\bf x}, {\bf y}), (f,\varphi):C\lra [\CC^5/\CC^*]\times [\CC^2/\CC^*]).$$ 
By the definition of stability, $\varphi$ is a regular map to $$\PP^1=\CC^2/\!\!/\CC^*\, $$ of class $1$.
Hence, the domain curve $C$ has a distinguished irreducible component $C_0$ canonically isomorphic to $\PP ^1$ via $\varphi$. 
The {\em standard} $\CC ^*$-action, 
\begin{equation}\label{tt44}
t\cdot [\xi _0, \xi _1] = [t\xi _0, \xi _1], \, \, \text{ for } t\in \CC ^*, \, [\xi _0, \xi _1]\in \PP ^1,
\end{equation}
induces  a $\CC ^*$-action on the graph space.

The $\CC^*$-equivariant cohomology of a point is
a free algebra with generator $z$,
$$H^*_{\CC ^*} (\Spec (\CC )) = \QQ [z]\, .$$
Our convention is to define
$z$ as the $\CC^*$-equivariant first Chern class of the tangent line $T_0\PP ^1$ at $0\in\PP^1$ with respect to the
action \eqref{tt44},
$$z=c_1(T_0\PP ^1)\, .$$

The $\T$-action on $\com^5$ lifts to a $\T$-action on
the graph space \eqref{xmmx} which commutes with the
$\CC^*$-action obtained from the distinguished domain component.
As a result, we have a $\T\times \CC^*$-action on the graph space
and
 $\T\times\CC ^*$-equivariant evaluation morphisms
\begin{align}
\nonumber     &\text{ev}_i: \mathsf{QG}_{g, n|k, \beta }^{0+, 0+} ([\CC^5/\CC^*] ) \ra       \PP^4 ,  & i=1,\dots,n\, ,\\
\nonumber &\widehat{\text{ev}}_j:  \mathsf{QG}_{g, n|k, \beta }^{0+,0+} 
([\CC^5/\CC^*] ) \ra       [\CC^5/\CC^*] , & j=1,\dots,k\, .
\end{align}
Since a morphism $$f: C \ra [\CC^5/\CC^*]$$ 
is equivalent to the data of a 
principal $\G$-bundle $P$ on $C$ and a section $u$ of $P\times _{\CC^*} 
\CC^5$, 
there is a natural morphism $$C\ra E\CC^* \times _{\CC^*} \CC^5$$ and hence a pull-back map
 \[ f^*:  H^*_{\CC^*}([\CC^5/\CC^*])  \ra H^*(C). \] 
 The above construction applied
 to the universal curve over the moduli space 
 and the universal morphism to $[\CC^5/\CC^*]$ is $\T$-equivariant.
 Hence,
 we obtain a pull-back map 
 \[ \widehat{\text{ev}}_j^*: H^*_{\T}(\CC^5, \QQ)\otimes_\QQ \QQ[z] \ra H^*_{\T\times \CC^*} (\mathsf{QG}_{g, n|k, \beta }^{0+, 0+} 
 ([\CC^5/\CC^*] ) , \QQ ) \]  
 associated to the evaluation map $\widehat{\text{ev}}_j$.
 


\subsubsection{{\em{I}}-functions}
The description of the fixed loci for the $\CC^*$-action 
on $$\mathsf{QG}_{g, 0|k, d }^{0+,0+} 
([\CC^5/\CC^*] )$$
is parallel to the description in \cite[\S4.1]{CKg0} for the unweighted case. 
In particular, there is a distinguished
subset $\F_{k,d}$ of the $\CC ^*$-fixed locus for which all the markings and the entire curve class $d$ lie  over $0 \in \PP ^1$. The locus
$\F_{k,d}$ comes with
a natural {\it proper} evaluation map $ev_{\bullet}$ obtained
from the generic point of $\PP ^1$:
\[ \text{ev}_\bullet:  \F_{k,d} \ra \CC^5/\!\!/\CC^* =\PP^4 .  \]

We can explicitly write
\begin{equation*}\F_{k,d}\cong \F_d\times 0^k\subset \F_d\times (\PP^1)^k,
\end{equation*}
where $\F_d$ is the $\CC^*$-fixed locus in $\mathsf{QG}^{0+}_{0,0,d}([\CC^5/\CC^*])$ for which the class $d$ is concentrated over $0\in\PP^1$.  The locus $\F_d$ parameterizes
quasimaps of class $d$,
$$f:\PP^1\lra [\CC^5/\CC^*]\, ,$$ with a base-point of 
length $d$ at $0\in\PP^1$. The restriction of $f$ to $\PP^1\setminus\{0\}$ is a constant map to $\PP^4$ defining the evaluation
map $\text{ev}_\bullet$.

As in \cite{CK, CKg0,CKM}, we define the big $\mathds{I}$-function as the generating function for
the push-forward via $ev_\bullet$ of localization residue contributions of $\F_{k,d}$.
For ${\bf t}\in  
H^*_{\T} ([\CC^5/\CC^*], \QQ )\ot _{\QQ} \QQ[z]$, let
 \begin{align*} \mathrm{Res}_{\F_{k,d}}({\bf t}^k) &=
 \prod_{j=1}^k \widehat{\text{ev}}_j^*({\bf t})\, \cap\, \mathrm{Res}_{\F_{k,d}}[
 \mathsf{QG}_{g, 0|k, d }^{0+,0+} 
([\CC^5/\CC^*])
 ]^{\mathrm{vir}} \\
 &=\frac{\prod_{j=1}^k \widehat{\text{ev}}_j^*({\bf t})
 \cap [\F_{k,d}]^{\mathrm{vir}}}
 {\mathrm{e}(\text{Nor}^{\mathrm{vir}}_{\F_{k,d}})},
 \end{align*}
 where 
$\text{Nor}^{\mathrm{vir}}_{\F_{k,d}}$ is the virtual normal bundle.

\begin{Def}\label{Je}
 The big $\mathds{I}$-function for the $(0+,0+)$-stability condition,
 as a formal function in $\bf t$,
 is
\begin{equation*}
\mathds{I}
(q,{\bf t}, z)=\sum_{d\geq 0}\sum_{k\geq 0} \frac{q^d}{k!}
\text{\em ev}_{\bullet\, *}\Big(\mathrm{Res}_{\F_{k,d}}({\bf t}^k)
\Big)\, .
\end{equation*} 
\end{Def} 

\subsubsection{Evaluations.} Let $\widetilde{H}\in H^*_{\mathsf{T}}([\CC^5/\CC^*])$ and $H\in H^*_{\mathsf{T}}(\PP^4)$ denote the respective hyperplane classes. The $\mathds{I}$-function of Definition \ref{Je} is evaluated in \cite{BigI}.

\begin{Prop}
For the restriction $\mathbf{t}=t \widetilde{H}\in H^*_{\mathsf{T}}([\CC^5/\CC^*],\QQ)$,
\begin{align*}
\mathds{I}(t)=\sum_{d=0}^\infty q^d e^{t(H+dz)/z}\frac{\prod_{k=0}^{5d}(5H+kz)}{\prod_{i=0}^{4}\prod_{k=1}^{d}(H-\lambda_i+kz)}\,.
\end{align*}
\end{Prop}

We return now to the functions $\mathds{S}_i(\gamma)$ defined in Section \ref{lightm}. Using Birkhoff factorization, an evaluation of the series $\mathds{S}(H^j)$ can be obtained from the $\mathds{I}$-function, see \cite{KL}:

\begin{align}\label{S1}
    \nonumber\mathds{S}(1)&=\frac{\mathds{I}}{\mathds{I}|_{t=0,H=1,z=\infty}}\,,\\
    \nonumber\mathds{S}(H)&=\frac{z\frac{d}{dt}\mathds{S}(1)}{z\frac{d}{dt}\mathds{S}(1)|_{t=0,H=1,z=\infty}}\, ,\\
    \mathds{S}(H^2)&=\frac{z\frac{d}{dt}\mathds{S}(H)}{z\frac{d}{dt}\mathds{S}(H)|_{t=0,H=1,z=\infty}}\, ,\\
    \nonumber\mathds{S}(H^3)&=\frac{z\frac{d}{dt}\mathds{S}(H^2)}{z\frac{d}{dt}\mathds{S}(H^2)|_{t=0,H=1,z=\infty}}\, ,\\
    \nonumber\mathds{S}(H^4)&=\frac{z\frac{d}{dt}\mathds{S}(H^3)}{z\frac{d}{dt}\mathds{S}(H^3)|_{t=0,H=1,z=\infty}}\, ,\\
    \nonumber\mathds{S}(1)&=\frac{z\frac{d}{dt}\mathds{S}(H^4)}{z\frac{d}{dt}\mathds{S}(H^4)|_{t=0,H=1,z=\infty}}\, .
\end{align}
For a series $F\in\CC[[\frac{1}{z}]]$, the specialization $F|_{z=\infty}$ denotes constant term of $F$ with respect to $\frac{1}{z}$.

\subsubsection{Further calculations.}\label{furcalc} Define small $I$-function
\begin{align*}
    \overline{\mathds{I}}(q)\in H^*_{\mathsf{T}}(\PP^4,\QQ)[[q]]
\end{align*}
by the restriction

\begin{align*}
    \overline{\mathds{I}}(q)=\mathds{I}(q,t)|_{t=0}\,.
\end{align*}
Define differential operators

\begin{align*}
    \mathsf{D}=q\frac{d}{dq}\,, \,\,\, M=H+z\mathsf{D}\,.
\end{align*}
Applying $z\frac{d}{dt}$ to $\mathds{I}$ and then restricting to $t=0$ has same effect as applying $M$ to $\overline{\mathds{I}}$
\begin{align*}
    \left[\left(z\frac{d}{dq}\right)^k \mathds{I} \right]|_{t=0}=M^k \overline{\mathds{I}}\,.
\end{align*}

The function $\overline{\mathds{I}}$ satisfies the following Picard-Fuchs equation
\begin{align*}
    \left(M^5-1-q(5M+z)(5M+2z)(5M+3z)(5M+4z)(5M+5z)\right)\overline{\mathds{I}}=0
\end{align*}
implied by the Picard-Fuchs equation for $\mathds{I}$,
\begin{align*}
    \left( \left(z\frac{d}{dt}\right)^5-1-q \prod_{k=1}^5\left(5\left( z\frac{d}{dt} \right) +kz\right)  \right) \mathds{I}=0\,.
\end{align*}

The restriction $\overline{\mathds{I}}|_{H=\lambda_i}$ admits the following asymptotic form
\begin{align}
 \label{assym}
 \overline{\mathds{I}}|_{H=\lambda_i}
 = e^{\mu\lambda_i/z}\left( R_0+R_1 \left(\frac{z}{\lambda_i}\right)+R_2 \left(\frac{z}{\lambda_i}\right)^2+\ldots\right)
\end{align}
with series 
$\mu,R_k \in \CC[[q]]$.

A derivation of \eqref{assym} is obtained in \cite{ZaZi} via the Picard-Fuchs equation 
for $\overline{\mathds{I}}|_{H=\lambda_i}$. The series $\mu$ and $R_k$ are found by solving differential equations obtained from the coefficient of $z^k$. For example,
\begin{eqnarray*}
    1+ \DD\mu&=& L\, , \\
    R_0&=&L\, , \\
    R_1&=& \frac{3}{20}(L-L^5)\, , \\ 
    R_2&=&\frac{9L}{800}(1-L^4)^2\, , 
\end{eqnarray*}
where $L(q) = (1-5^5q)^{-1/5}$. The specialization \eqref{spez}
is used for these results.

Define the series $C_i$ by the equations
\begin{align}\label{y999}
C_0 & = \mathds{I}|_{z=\infty,t=0,H=1}\,,\\
C_i & = z\frac{d}{dt} \mathds{S}({H^{i-1}})|_{z=\infty,t=0,H=1}\, , \,\,\, \text{for}\,i=1,2,3,4\, .
\end{align}
The following relations were proven in \cite{ZaZi},
\begin{align*}
    C_0C_1C_2C_3C_4&=L^5\,,\\
    C_i&=C_{4-i}\,\,\, \text{for}\,\,\,i=0,1,2,3,4.
\end{align*}
From the equations \eqref{S1} and \eqref{assym}, we can show the series
$$\overline{\mathds{S}}_i(H^k)=\mathds{S}(H^k)|_{H=\lambda_i,t=0}$$
have the following asymptotic expansion:
\begin{align}
\nonumber
\overline{\mathds{S}}_i({1}) & = e^{\frac{\mu\lambda_i}{z}}\frac{1}{C_0} \Big(R_{00}+R_{01}\big(\frac{z}{\lambda_i}\big)+R_{02} \big(\frac{z}{\lambda_i}\big)^2+\ldots\Big) \, ,\\  \label{VS}
\overline{\mathds{S}}_i(H) & = e^{\frac{\mu\lambda_i}{z}} \frac{L\lambda_i}{C_0C_1} \Big(R_{10}+R_{11}\big(\frac{z}{\lambda_i}\big)+R_{12} \big(\frac{z}{\lambda_i}\big)^2+\ldots\Big)\, ,\\ \nonumber
\overline{\mathds{S}}_i(H^2) & = e^{\frac{\mu\lambda_i}{z}} \frac{L^2\lambda_i^2}{C_0C_1 C_2} \Big(R_{20}+R_{21}\big(\frac{z}{\lambda_i}\big)+R_{22} \big(\frac{z}{\lambda_i}\big)^2+\ldots\Big)\, ,\\ \nonumber
\overline{\mathds{S}}_i(H^3) & = e^{\frac{\mu\lambda_i}{z}} \frac{L^3\lambda_i^3}{C_0C_1 C_2C_3}\Big(R_{30}+R_{31}\big(\frac{z}{\lambda_i}\big)+R_{32} \big(\frac{z}{\lambda_i}\big)^2+\ldots\Big)
\, ,\\ \nonumber
\overline{\mathds{S}}_i(H^4) & = e^{\frac{\mu\lambda_i}{z}} \frac{L^4\lambda_i^4}{C_0C_1 C_2C_3C_4}\Big(R_{40}+R_{41}\big(\frac{z}{\lambda_i}\big)+R_{42} \big(\frac{z}{\lambda_i}\big)^2+\ldots\Big)\, .
 \end{align}
We follow here the normalization of \cite{ZaZi}. Note 
$$R_{0k}=R_k\,.$$
As in \cite[Theorem 4]{ZaZi}, we obtain the following constraints.
\begin{Prop}
{\em (Zagier-Zinger \cite{ZaZi})}\label{RPoly}
 For all $k\geq 0$, we have
     $$R_k \in \CC[L^{\pm1}]\, .$$
\end{Prop}

Define generators
$$    \mathcal{X} = \frac{\mathsf{D}C_0}{C_0}\,,\ \    \mathcal{X}_1 = \mathsf{D} \mathcal{X}\,,\ \  
    \mathcal{X}_2 = \mathsf{D}\mathcal{X}_1\,,\ \ 
    \mathcal{Y} =\frac{\mathsf{D}C_1}{C_1}\,.$$
From \eqref{S1}, we obtain the following result.
\begin{Lemma}\label{RR} For $k\ge 0$ wee have
  \begin{align}
      \nonumber R_{1\, k+1}= &R_{0\, k+1}+\frac{\mathsf{D}R_{0 k}}{L}-\frac{\mathcal{X}}{L}R_{0k}\,, \\
      \nonumber R_{2\, k+1}= &R_{1\, k+1}+\frac{\mathsf{D}R_{1k}}{L}-\frac{\mathcal{X}}{L}R_{1k}-\frac{\mathcal{Y}}{L}R_{1k}+\frac{\mathsf{D}L}{L^2}R_{1k}\,, \\
      R_{3\, k+1}= &R_{2\, k+1}+\frac{\mathsf{D}R_{2k}}{L}+\frac{\mathcal{X}}{L}R_{2k}+\frac{\mathcal{Y}}{L}R_{2k}-3\frac{\mathsf{D}L}{L^2}R_{2k}\,, \\
      \nonumber R_{4\, k+1}= &R_{3\, k+1}+\frac{\mathsf{D}R_{3k}}{L}+\frac{\mathcal{X}}{L}R_{3k}-2\frac{\mathsf{D}L}{L^2}R_{3k}\,, \\
      \nonumber R_{0\,k+1}=&R_{4\,k+1}+\frac{\mathsf{D}R_{4k}}{L}-\frac{\mathsf{D}L}{L^2}R_{4k}\,.
  \end{align}
\end{Lemma}

Applying Lemma \ref{RR} for $k=0,1$, we obtain the following  two equations among above generators which were also proven in \cite[Section 3.1]{YY}. First,
\begin{eqnarray}\label{drule}
    \mathsf{D}\mathcal{Y}&=& \frac{2}{5}(L^5-1)+2(L^5-1)\mathcal{X}-2\mathcal{X}^2-4\mathcal{X}_1\\ & &
    \nonumber +(L^5-1)\mathcal{Y}-\mathcal{Y}^2
-2\mathcal{X}\mathcal{Y}  \, .
\end{eqnarray}
For the second equation, define{\footnote{We follow here the notation of  \cite{YY} for $B_k$.}}
\begin{align*}
    B_1&=-5\mathcal{X}\,,\\
    B_2&=5^2(\mathcal{X}_1+\mathcal{X}^2)\,,\\
    B_3&=-5^3(\mathcal{X}_2+3\mathcal{X}\mathcal{X}_1+\mathcal{X}^3)\,,\\
    B_4&=5^4(\mathsf{D}\mathcal{X}_2+4\mathcal{X}\mathcal{X}_2+3\mathcal{X}_1^2+6\mathcal{X}^2\mathcal{X}_1+\mathcal{X}^4)\,.
\end{align*}
Then, we have
\begin{align}\label{drule2}
    B_4=-(L^5-1)(10B_3-35B_2+50B_1-24)\, .
\end{align}

For the proof of first holomorphic anomaly equation, we will
require the following generalization of Proposition \ref{RPoly}.
\begin{Prop}\label{RPoly2}
For all $k \ge 0$, we have
\begin{itemize}\label{Rpoly}
    \item[(i)] $R_{1k}\in \CC[L^{\pm 1}][\mathcal{X}]$ ,
    \item[(ii)]  $R_{2k}=Q_{2k}-\frac{R_{1\,k-1}}{L}\mathcal{Y}$, with
    $Q_{2k}\in \CC[L^{\pm 1}][\mathcal{X},\mathcal{X}_1]$ ,
    \item[(iii)] $R_{3k},R_{4k} \in \CC[L^{\pm 1}][\mathcal{X},\mathcal{X}_1,\mathcal{X}_2]$ .
\end{itemize}
\end{Prop}

\begin{proof}

 \begin{itemize}
     \item [(i)] Using Lemma \ref{RR}, we can calculate
     $$R_{1\,k+1}=\frac{\mathsf{D}R_{0k}}{L}+R_{0\,k+1}-\frac{R_{0k}}{L}\mathcal{X}\,.$$
     \item [(ii)] Using Lemma \ref{RR} and relations \eqref{drule}, we can calculate
     \begin{multline*}
         R_{2\,k+2}=\frac{\mathsf{D}^2R_{0k}}{L^2}-\frac{R_{0\,k+1}}{5L}+\frac{L^4 R_{0\,k+1}}{5}+\frac{2\mathsf{D}R_{0\,k+1}}{L}+R_{0\,k+2}\\-\frac{2\mathsf{D}R_{0k}\mathcal{X}}{L^2}-\frac{2 R_{0\,k+1}\mathcal{X}}{L}+\frac{R_{0k}\mathcal{X}^2}{L^2}-\frac{R_{0k}\mathcal{X}_1}{L^2}\\
         +\frac{-\mathsf{D}R_{0k}-LR_{0\,k+1}+R_{0k}\mathcal{X}}{L^2}\mathcal{Y}\,.
     \end{multline*}
     \item[(iii)] We can also explicitly calculate $R_{3k}$ and $R_{4k}$ in terms of $$R_{0k}\,,\ R_{0\,k-1}\,,\ R_{0\,k-2}\,,\ 
     \mathcal{X}\,,\ \mathcal{X}_1\,,\ \mathcal{X}_2\,,\  \mathcal{Y}$$ using Lemma \ref{RR} and relations \eqref{drule} and \eqref{drule2}. We can check (iii) using these explicit calculations and Proposition \ref{RPoly}. We leave the details to the reader.
 \end{itemize}
\end{proof}

For the proof of second holomorphic anomaly equation, we 
will require the following result.

\begin{Prop}\label{RPoly3}
 For all $k\ge 0$, we have
 \begin{itemize}
     \item[(i)] $R_{1k}=P_{0k}-\frac{R_{0\,k-1}}{L}\mathcal{X}$, with $P_{0k}\in \CC[L^{\pm 1}]$,
     \item[(ii)] $R_{2k}\in \CC[L^{\pm 1}][A_2,A_4]$,
     \item[(iii)] $R_{3k}=P_{3k}-\frac{R_{2\,k-1}}{L}\mathcal{X}$ with $P_{3k}\in\CC[L^{\pm 1}][A_2,A_4,A_6]$,
     \item[(iv)] $R_{4k}\in\CC[L^{\pm 1}][A_2,A_4,A_6]$.
 \end{itemize}
\end{Prop}

\begin{proof}
 The proof follows from the explicit calculations in the proof of Proposition \ref{RPoly2} and the definition of $A_2,A_4,A_6$.
\end{proof}

\section{Higher genus series on $\overline{M}_{g,n}$}\label{hgi}

 \subsection{Intersection theory on $\overline{M}_{g,n}$} \label{intmg}
 We review here the now standard method used by Givental \cite{Elliptic,SS,Book} to 
 express genus $g$ descendent correlators in terms of genus 0 data.
 We refer the reader to \cite[Section 4.1]{LP1}
 for a more leisurely treatment.
 
 Let $t_0,t_1,t_2, \ldots$ be formal variables. The series
 $$T(c)=t_0+t_1 c+t_2 c^2+\ldots$$   in the
additional variable $c$ plays a basic role. The variable $c$
will later be  replaced by the first Chern class $\psi_i$ of
 a cotangent line  over $\overline{M}_{g,n}$, 
 $$T(\psi_i)= t_0 + t_1\psi_i+ t_2\psi_i^2 +\ldots\, ,$$
 with the index $i$
 depending on the position of the series $T$ in the correlator.

Let $2g-2+n>0$.
For $a_i\in \mathbb{Z}_{\geq 0}$ and  $\gamma \in H^*(\overline{M}_{g,n})$, define the correlator 
\begin{multline*}
    \lann \psi^{a_1},\ldots,\psi^{a_n}\, | \, \gamma\,  \rann_{g,n}=
    \sum_{k\geq 0} \frac{1}{k!}\int_{\overline{M}_{g,n+k}}
    \gamma \, \psi_1^{a_1}\cdots 
     \psi_n^{a_n}  \prod_{i=1}^k T(\psi_{n+i})\, . 
\end{multline*}
Here, $\gamma$ also denotes the pull-back of $\gamma$ via the morphism
$$\overline{M}_{g,n+k}\rightarrow\overline{M}_{g,n}\,$$
defined by forgetting the last $k$ points.
In the above summation,
the $k=0$ term is $$\int_{\overline{M}_{g,n}}\gamma\, \psi_1^{a_1}\cdots\psi_n^{a_n}\,.$$
We also need the following correlator defined for the unstable case,

$$\lan\lan 1,1 \ran\ran_{0,2}=\sum_{k > 0}\frac{1}{k!}\int_{\overline{M}_{0,2+k}}\prod_{i=1}^k T(\psi_{2+i})\,.$$

For formal variables $x_1,\ldots,x_n$, we also define the correlator
\begin{align}\label{derf}
\lannn \frac{1}{x_1-\psi},\ldots,\frac{1}{x_n-\psi}\, \Big| \, \gamma \, \rannn_{g,n}
\end{align}
in the standard way by expanding $\frac{1}{x_i-\psi}$ as a geometric series.

Denote by $\mathds{L}$ the differential operator 
\begin{align*}
        \mathds{L}\, =\, 
        \frac{\partial}{\partial t_0}-\sum_{i=1}^\infty t_i\frac{\partial}{\partial t_{i-1}}
        \, =\, \frac{\partial}{\partial t_0}-t_1\frac{\partial}{\partial t_0}-t_2\frac{\partial}{\partial t_1}-\ldots
        \, .
\end{align*}
 The string equation yields the following result.
 
\begin{Lemma} \label{stst} For $2g-2+n>0$ and $\gamma \in H^*(\overline{M}_{g,n})$, we have
$$\mathds{L}\lann 1,\ldots,1\, | \, \gamma\, \rann_{g,n}=0\, ,$$ 
\begin{multline*}
\mathds{L} \lannn \frac{1}{x_1-\psi},\ldots,\frac{1}{x_n-\psi}\, \Big| \,\gamma \, 
\rannn_{g,n}= \\
 \left(\frac{1}{x_1}+\ldots +\frac{1}{x_n}\right)
 \lannn\frac{1}{x_1-\psi},\ldots \frac{1}{x_n-\psi}\, \Big| \, \gamma \, \rannn_{g,n}\, .
 \end{multline*}
\end{Lemma}


We consider 
$\CC(t_1)[t_2,t_3,...]$
as $\ZZ$-graded ring over $\CC(t_1)$ with 
$$\text{deg}(t_i)=i-1\ \ \text{for $i\geq 2$ .}$$
Define a subspace of homogeneous elements by
$$\CC\left[\frac{1}{1-t_1}\right][t_2,t_3,\ldots]_{\text{Hom}} \subset 
\CC(t_1)[t_2,t_3,...]\, .
$$
After the restriction $t_0=0$ and application of the dilaton equation, the correlators are expressed in terms of finitely many integrals  (by the dimension constraints). From this,
we easily see 
$$\lann \psi^{a_1},\ldots,\psi^{a_n}\, | \, \gamma \, \rann_{g,n}\, |_{t_0=0}\ \in\
\CC\left[\frac{1}{1-t_1}\right][t_2,t_3,\ldots]_{\text{Hom}}\, .$$
Using the leading terms (of lowest degree in $\frac{1}{(1-t_1)}$), we obtain the
following result.

\begin{Lemma}\label{basis}
The set of genus 0 correlators
 $$
 \Big\{ \, \lann 1,\ldots,1\rann_{0,n}\, |_{t_0=0} \, \Big\}_{n\geq  4} $$ 
freely generate the ring
 $\CC(t_1)[t_2,t_3,...]$ over $\CC(t_1)$.
\end{Lemma}


\begin{Def} 
For $\gamma \in H^*(\overline{M}_{g,k})$, let $$\pP^{a_1,\ldots,a_n,\gamma}_{g,n}(s_0,s_1,s_2,...)\in \QQ(s_0, s_1,..)$$ be 
the unique rational function satisfying the condition
$$\lann \psi^{a_1},\ldots,\psi^{a_n}\, |\, \gamma\, \rann_{g,n}|_{t_0=0}
=\pP^{a_1,a_2,...,a_n,\gamma}_{g,n}|_{s_i=\lann 1,\ldots,1\rann_{0,i+3}|_{t_0=0}}\, . $$
\end{Def}
 
By applying Lemma \ref{stst}, we obtain the two following
results, see \cite[Section 4.1]{LP1}.
 
\begin{Prop}\label{GR1} For $2g-2+n>0$,
we have
 $$\lann 1,\ldots,1\,|\, \gamma\, \rann_{g,n}
=\pP^{0,\ldots,0,\gamma}_{g,n}|_{s_i=\lann 1,\ldots,1\rann_{0,i+3}}\, . $$
\end{Prop} 


\begin{Prop}\label{GR2} For $2g-2+n>0$,
 \begin{multline*}
     \lannn \frac{1}{x_1-\psi_1}, \ldots, \frac{1}{x_n-\psi_n}\, \Big| \, \gamma \, \rannn_{g,n}= \\
     e^{\lann 1,1\rann_{0,2}(\sum_i\frac{1}{x_i})}\sum_{a_1,\ldots,a_n}\frac{\pP^{a_1,\ldots,a_n,\gamma}_{g,n}|_{s_i=\lann 1,\ldots,1\rann_{0,i+3}}
     }{x_1^{a_1+1} \cdots x_n^{a_n+1}}.
 \end{multline*}
\end{Prop} 
 
     $$\mathds{L} \lann 1,1\rann_{0,2} =1\,, \ \ \ \  \lann 1,1\rann_{0,2}|_{t_0=0}=0\, .$$

The definition given in \eqref{derf}
of the correlator is valid
in the stable range $$2g-2+n>0\, .$$
The unstable case $(g,n)=(0,2)$ plays a
special role. We define
$$\lannn \frac{1}{x_1-\psi_1}, \frac{1}{x_2-\psi_2}\rannn_{0,2}$$
by 
adding the
degenerate term
$$\frac{1}{x_1+x_2}$$
to the terms obtained
by the 
 expansion of $\frac{1}{x_i-\psi_i}$ as 
 a geometric series.
 The degenerate term is associated
to the (unstable) moduli space
of genus 0 with 2 markings.
By \cite[Section 4.2]{LP1}, we have.

\begin{Prop}\label{GR22} We have
 \begin{equation*}
     \lannn \frac{1}{x_1-\psi_1}, \frac{1}{x_2-\psi_2} \rannn_{0,2}= 
     e^{\lann 1,1\rann_{0,2}\left(\frac{1}{x_1}+
     \frac{1}{x_2}\right)}\left(\frac{1}{x_1+x_2}\right)\, .
 \end{equation*}
\end{Prop} 
 

\subsection{Local invariants and wall crossing} 
The torus $\T$ acts on the moduli spaces
$\overline{M}_{g,n}(\PP^4,d)$  and
$\overline{Q}_{g,n}(\PP^4,d)$.
We consider here special localization contributions 
associated to the fixed points ${p}_i\in \PP^4$.


Consider first the moduli of stable maps.
Let
$$\overline{M}_{g,n}(\PP^4,d)^{\T,p_i}
\subset \overline{M}_{g,n}(\PP^4,d) $$
be the union of
 $\T$-fixed loci which parameterize stable maps
obtained by attaching $\T$-fixed rational tails to a genus $g$, $n$-pointed
Deligne-Mumford stable curve contracted
to the point $p_i\in\PP^4$.
Similarly, let 
$$\overline{Q}_{g,n}(\PP^4,d)^{\T,p_i}\subset
\overline{Q}_{g,n}(\PP^4,d)
$$
be the parallel $\T$-fixed locus
parameterizing stable quotients obtained
by attaching base points
to  a genus $g$, $n$-pointed
Deligne-Mumford stable curve contracted
to the point $p_i\in\PP^4$.

Let $\Lambda_i$ denote the localization of the ring
$$\CC[\lambda^{\pm 1}_0,\dots,\lambda^{\pm 1}_4]$$ at 
the five tangent weights at $p_i\in \PP^4$.
Using the virtual
localization formula \cite{GP}, 
there exist unique series
$$S_{p_i}\in\Lambda_i[\psi][[Q]]$$ 
for which the localization contribution 
of the $\T$-fixed locus
$\overline{M}_{g,n}(\PP^4,d)^{\T,p_i}$
to the equivariant Gromov-Witten
invariants of formal quintic
can be written as
\begin{multline*}
    \sum_{d=0}^\infty Q^d \int_{[\overline{M}_{g,n}(\PP^4,d)^{\T,p_i}]^{\vir}}\frac{e(\text{Obs})}{e(\text{Nor})}
    \psi_1^{a_1}\cdots\psi_n^{a_n}=\\
    \sum_{k=0}^\infty \frac{1}{k!} \int_{\overline{M}_{g,n+k}}
    {\mathsf{H}}_{g}^{p_i}\, \psi_1^{a_1}\cdots\psi_n^{a_n}\, \prod_{j=1}^k S_{p_i}(\psi_{n+j})\, .
\end{multline*}
Here, $\mathsf{H}_{g}^{p_i}$ is the standard vertex class, 
\begin{equation}\label{hhbb}
\frac{e(\mathbb{E}_g^*\otimes T_{p_i}(\PP^4)))}{e(T_{p_i}(\PP^4)} \cdot \frac{(5\lambda_i)}{e(\mathbb{E}_g^* \otimes(5\lambda_i))}\, ,
\end{equation}
obtained from the  Hodge bundle $\mathbb{E}_g\rightarrow \overline{M}_{g,n+k}$.

Similarly, the application of the
virtual localization formula to the moduli of stable
quotients yields classes
$$F_{p_i,k}\in H^*(\overline{M}_{g,n|k})\otimes_\CC\Lambda_i$$ 
for which the contribution of $\overline{Q}_{g,n}(\PP^4,d)^{T,p_i}$ is given by
\begin{multline*}
    \sum_{d=0}^\infty q^d \int_{[\overline{Q}_{g,n}(\PP^4,d)^{\T,p_i}]^{\vir}}\frac{e(\text{Obs})}{e(\text{Nor})}\psi_1^{a_1}\cdots
    \psi_n^{a_n}=\\
    \sum_{k=0}^\infty \frac{q^k}{k!} \int_{\overline{M}_{g,n|k}} \mathsf{H}_{g}^{p_i}\, \psi_1^{a_1}\cdots \psi_n^{a_n}\, F_{p_i,k}.
\end{multline*}
Here $\overline{M}_{g,n|k}$ is the moduli space of genus $g$ curves with markings
$$\{p_1,\cdots,p_n\}\cup\{\hat{p}_1\cdots\hat{p}_k \}\in C^{\text{ns}}\subset C$$
satisfying the conditions
\begin{itemize}
 \item[(i)] the points $p_i$ are distinct,
 \item[(ii)] the points $\hat{p}_j$ are distinct from the points $p_i$,
\end{itemize}
with stability given by the ampleness of 
$$\omega_C(\sum_{i=1}^m p_i+\epsilon\sum_{j=1}^k \hat{p}_j)$$
for every strictly positive $\epsilon \in \QQ$.

The Hodge class $\mathsf{H}_{g}^{p_i}$ is given again by
formula \eqref{hhbb} using the Hodge bundle $$\mathbb{E}_g\rightarrow \overline{M}_{g,n|k}\, .$$

\begin{Def}
 For $\gamma\in H^*(\overline{M}_{g,n})$, let
 \begin{eqnarray*}
     \lann \psi_1^{a_1},\ldots,\psi_n^{a_n}\, |\, \gamma\, \rann_{g,n}^{p_i,\infty}
     &=&
     \sum_{k=0}^\infty \frac{1}{k!}
     \int_{\overline{M}_{g,n+k}} \gamma \, \psi_1^{a_1}\cdots \psi_n^{a_n}\prod_{j=1}^k S_{p_i}(\psi_{n+j})\, ,\\
    \lann \psi_1^{a_1},\ldots,\psi_n^{a_n}\, |\, \gamma\, \rann_{g,n}^{p_i,0+}&=&
    \sum_{k=0}^\infty \frac{q^k}{k!} \int_{\overline{M}_{g,n|k}} \gamma \, \psi_1^{a_1}\cdots \psi_n^{a_n}\, F_{p_i,k}\, .
 \end{eqnarray*}
\end{Def}

\
\begin{Prop} [Ciocan-Fontanine, Kim \cite{CKg}] \label{WC} 
For $2g-2+n>0$,
we have the wall crossing relation
$$\lann \psi_1^{a_1},\ldots,\psi_n^{a_n}\, |\, \gamma\, \rann_{g,n}^{p_i,\infty}(Q(q))= (I^{\mathsf{Q}}_0)^{2g-2+n}\lann \psi_1^{a_1},\ldots,\psi_n^{a_n}\, |\, \gamma\rann_{g,n}^{p_i,0+}(q)$$
 where 
 $Q(q)$ is the mirror map
 $$Q(q)=\exp\left(\frac{I_1^{\mathsf{Q}}(q)}{I_0^{\mathsf{Q}}(q)}\right)\, .$$
\end{Prop}

Proposition \ref{WC} is a consequence
of \cite[Lemma 5.5.1]{CKg}. The mirror
map here is the mirror map for
quintic discussed in Section \ref{holp}.
 Propositions \ref{GR1} and \ref{WC} together yield 
 \begin{eqnarray*}
 \lann 1,\ldots,1\,  |\, \gamma \, \rann_{g,n}^{p_i,\infty}& =& \pP^{0,\ldots,0,\gamma}_{g,n}\big(\lann 1,1,1\rann_{0,3}^{p_i,\infty},\lann 1,1,1,1\rann _{0,4}^{p_i,\infty},\ldots\big)\, ,\\
 \lann 1,\ldots,1\,  |\, \gamma \, \rann_{g,n}^{p_i,0+}&=&\pP^{0,\ldots,0,\gamma}_{g,n}\big(\lann
 1,1,1\rann_{0,3}^{p_i,0+},\lann 1,1,1,1\rann _{0,4}^{p_i,0+},\ldots\big)\, .
 \end{eqnarray*}
Similarly, using Propositions \ref{GR2} and \ref{WC}, we obtain
\begin{multline*}
\lannn \frac{1}{x_1-\psi}, \ldots, \frac{1}{x_n-\psi}\, \Big| \, \gamma \, \rannn_{g,n}^{p_i,\infty}= \\
     e^{\lann 1,1\rann^{p_i,\infty}_{0,2}\left(\sum_i\frac{1}{x_i}\right)}\sum_{a_1,\ldots,a_n}\frac{\pP^{a_1,\ldots,a_n,\gamma}_{g,n}\big(\lann 1,1,1\rann_{0,3}^{p_i,\infty},\lann 1,1,1,1\rann_{0,4}^{p_i,\infty},\ldots \big)}{x_1^{a_1+1}\cdots x_n^{a_n+1}}\, ,
\end{multline*}
\begin{multline}\label{ppqqpp}
\lannn \frac{1}{x_1-\psi}, \ldots, \frac{1}{x_n-\psi}\, \Big| \, \gamma \, \rannn_{g,n}^{p_i,0+}= \\
     e^{\lann 1,1\rann^{p_i,0+}_{0,2}\left(\sum_i\frac{1}{x_i}\right)}\sum_{a_1,\ldots,a_n}\frac{\pP^{a_1,\ldots,a_n,\gamma}_{g,n}\big(\lann 1,1,1\rann_{0,3}^{p_i,0+},\lann 1,1,1,1\rann_{0,4}^{p_i,0+},\ldots \big)}{x_1^{a_1+1}\cdots x_n^{a_n+1}}\, .
\end{multline}

\section{Higher genus series on the formal quintic}\label{hgs}
\subsection{Overview}
We apply Givental's the localization strategy    \cite{Elliptic,SS,Book} for Gromov-Witten theory to the stable quotient invariants of formal quintic. 
The contribution $\text{Cont}_\Gamma(q)$ 
discussed in Section \ref{locq} 
of a graph $\Gamma \in \mathsf{G}_{g}(\PP^4)$ 
can be separated into vertex and edge contributions.
We express the vertex and edge contributions in terms of
the series $\mathds{S}_i$ and $\mathds{V}_{ij}$ of Section \ref{lightm}.
Our treatment here follows our study of
$K\proj^2$ in \cite[Section 5]{LP1}

\subsection{Edge terms}
Recall the definition{\footnote{We use
the variables $x_1$ and $x_2$ here instead
of $x$ and $y$.}}of $\mathds{V}_{ij}$
given in Section \ref{lightm},
\begin{equation}\label{dfdf6}
\mathds{V}_{ij}  =  
\Big\langle \Big\langle  \frac{\phi _i}{x- \psi } ,  \frac{\phi _j}{y - \psi } 
\Big\rangle \Big\rangle  _{0, 2}^{0+,0+}  \, .
\end{equation}
Let $\overline{\mathds{V}}_{ij}$ denote
the restriction of $\mathds{V}_{ij}$
to $t=0$.
Via formula \eqref{ddgg},
$\overline{\mathds{V}}_{ij}$ is a summation of contributions of fixed loci indexed by
a graph $\Gamma$ consisting of two vertices 
connected by a unique edge. 
Let $w_1$ and $w_2$ be 
$\T$-weights. Denote by $$\overline{{\mathds{V}}}_{ij}^{w_1,w_2}$$ the summation of contributions of $\T$-fixed loci with
tangent weights precisely $w_1$
and $w_2$ on the first rational components
which exit the vertex components over
$p_i$ and $p_j$.

The series $\overline{{\mathds{V}}}_{ij}^{w_1,w_2}$
includes {\em both} vertex and edge
contributions.
By definition \eqref{dfdf6} and the virtual localization formula, we find the
following relationship between
$\overline{\mathds{V}}_{ij}^{w_1,w_2}$
and the corresponding
pure edge contribution $\mathsf{E}_{ij}^{w_1,w_2}$,

\begin{eqnarray*}
    e_i\overline{\mathds{V}}_{ij}^{w_1,w_2}e_j
    &=& \lannn \frac{1}{w_1-\psi},\frac{1}{x_1-\psi}\rannn^{p_i,0+}_{0,2}\mathsf{E}_{ij}^{w_1,w_2}
    \lannn \frac{1}{w_2-\psi},\frac{1}{x_2-\psi}\rannn^{p_j,0+}_{0,2}\\
    &=&\frac{e^{\frac{\lann 1,1\rann^{p_i,0+}_{0,2}}{w_1}+\frac{\lann 1,1\rann^{p_i,0+}_{0,2}}{x_1}}}{w_1+x_1}
    \, \mathsf{E}^{w_1,w_2}_{ij}\, \frac{e^{\frac{\lann 1,1\rann^{p_j,0+}_{0,2}}{w_2}+\frac{\lann 1,1\rann^{p_j,0+}_{0,2}}{x_2}}}{w_2+x_2}
     \end{eqnarray*}
    
\begin{align*}        
    =\sum_{a_1,a_2}e^{\frac{\lann 1,1\rann^{p_i,0+}_{0,2}}{x_1}+\frac{\lann 1,1\rann^{p_i,0+}_{0,2}}{w_1}}e^{\frac{\lann 1,1\rann^{p_j,0+}_{0,2}}{x_2}+\frac{\lann 1,1\rann^{p_j,0+}_{0,2}}{w_2}}(-1)^{a_1+a_2} \frac{
    \mathsf{E}^{w_1,w_2}_{ij}}{w_1^{a_1}w_2^{a_2}}x_1^{a_1-1}x_2^{a_2-1}\, .
\end{align*}
After summing over all possible weights, we obtain
$$
    e_i\left(\overline{\mathds{V}}_{ij}-\frac{\delta_{ij}}{e_i(x_1+x_2)}\right)e_j=\sum_{w_1,w_2} e_i\overline{\mathds{V}}_{ij}^{w_1,w_2}e_j\, .$$
The above calculations immediately yield
the following result.
    

\begin{Lemma}\label{Edge} We have
 \begin{multline*}
 \left[e^{-\frac{\lann1,1\rann^{p_i,0+}_{0,2}}{x_1}}
       e^{-\frac{\lann1,1\rann^{p_j,0+}_{0,2}}{x_2}}e_i\left(\overline{\mathds{V}}_{ij}-\frac{\delta_{ij}}{e_i(x_1+x_2)}\right)e_j\right]_{x_1^{a_1-1}x_2^{a_2-1}}=\\
       \sum_{w_1,w_2}
       e^{\frac{\lann1,1\rann^{p_i,0+}_{0,2}}{w_1}}e^{\frac{\lann1,1\rann^{p_j,0+}_{0,2}}{w_2}}(-1)^{a_1+a_2}\frac{\mathsf{E}_{ij}^{w_1,w_2}}{w_1^{a_1}w_2^{a_2}}\, .
 \end{multline*}
\end{Lemma}

\noindent The notation $[\ldots]_{x_1^{a_1-1}x_2^{a_2-1}}$
in Lemma \ref{Edge} denotes the coefficient of
 $x_1^{a_1-1}x_2^{a_2-1}$ in the series expansion
 of the argument.

\subsection{A simple graph}\label{simgr}
Before treating the general case, we present
the localization formula for a simple graph{\footnote{We
follow here the notation of Section \ref{locq}.}.
Let
$\Gamma\in \mathsf{G}_{g}(\PP^4)$ 
consist of two vertices   and one edge,
$$v_1,v_2\in \Gamma(V)\, , \ \ \ \ 
e\in \Gamma(E)\, $$
with genus and $\T$-fixed point assignments
$$\mathsf{g}(v_i)=g_i\, , \ \ \ \ \mathsf{p}(v_i)=p_i\, .$$

Let $w_1$ and $w_2$ be tangent
weights at the vertices $p_1$ and $p_2$
respectively. Denote by $\text{Cont}_{\Gamma,w_1,w_2}$
the summation of contributions to
\begin{equation}\label{zlzl}
\sum_{d>0} q^d\,( e(\text{Obs})\cap\left[\overline{Q}_{g}(\PP^4,d)\right]^{\vir})
\end{equation}
of $\T$-fixed loci with
tangent weights precisely $w_1$
and $w_2$ on the first rational components
which exit the vertex components over
$p_1$ and $p_2$.
We can express the localization formula for 
\eqref{zlzl} as
$$
\lannn \frac{1}{w_1-\psi}\, \Big|\, \mathsf{H}_{g_1}^{p_1}
\rannn_{g_1,1}^{p_1,0+}
\mathsf{E}^{w_1,w_2}_{12} \lannn\frac{1}{w_2-\psi}\, \Big|\, \mathsf{H}_{g_2}^{p_2}
\rannn_{g_2,1}^{p_2,0+} $$
which equals
$$\sum_{a_1,a_2} e^{\frac{\lann1,1\rann^{p_1,0+}_{0,2}}{w_1}}\frac{\ppl
{\psi^{a_1-1}} \, \Big|\, \mathsf{H}_{g_1}^{p_1}     \ppr_{g_1,1}^{p_1,0+}} {w_1^{a_1}} \mathsf{E}^{w_1,w_2}_{12} e^{\frac{\lann 1,1\rann^{p_2,0+}_{0,2}}{w_2}}\frac{\ppl {\psi^{a_2-1}} \, \Big|\, \mathsf{H}_{g_2}^{p_2}\ppr_{g_2,1}^{p_2,0+}}{w_2^{a_2}}
$$
where $\mathsf{H}_{g_i}^{p_i}$ is
the Hodge class \eqref{hhbb}. We have used here
the notation
\begin{multline*}
\ppl
\psi^{k_1}_1, \ldots,\psi^{k_n}_n \, \Big|\, \mathsf{H}_{h}^{p_i}     \ppr_{h,n}^{p_i,0+} 
=\\
\pP^{k_1,\ldots,k_n,\mathsf{H}_{h}^{p_i}  }_{h,1}\big(\lann 1,1,1\rann_{0,3}^{p_i,0+},\lann 1,1,1,1\rann_{0,4}^{p_i,0+},\ldots \big)
\,
\end{multline*}
and applied \eqref{ppqqpp}.

After summing over all possible weights $w_1,w_2$ and
applying 
Lemma \ref{Edge}, we obtain the following result for the full contribution $$\text{Cont}_\Gamma = \sum_{w_1,w_2} \text{Cont}_{\Gamma,w_1,w_2}$$
of $\Gamma$ to $\sum_{d\geq 0} q^d (e(\text{Obs})\cap\left[ \overline{Q}_{g}(\PP^4,d)\right]^{\vir})$.

\begin{Prop} We have \label{propsim}
 \begin{multline*}
     \text{\em Cont}_{\Gamma}=
     \sum_{a_1,a_2>0}
     \ppl
{\psi^{a_1-1}}  \, \Big|\, \mathsf{H}_{g_1}^{p_i}\,     \ppr_{g_1,1}^{p_i,0+}
\ppl
{\psi^{a_2-1}}  \, \Big|\, \mathsf{H}_{g_2}^{p_j}\,     \ppr_{g_2,1}^{p_j,0+}\ \ \ \ \ \ \ \ \ \ \ \\
\ \ \ \ \ \ \ \ \ \ \cdot
     (-1)^{a_1+a_2}\left[e^{-\frac{\lann1,1\rann^{p_i,0+}_{0,2}}{x_1}}
       e^{-\frac{\lann1,1\rann^{p_j,0+}_{0,2}}{x_2}}e_i\left(\overline{\mathds{V}}_{ij}-\frac{\delta_{ij}}{e_i(x_1+x_2)}\right)e_j\right]_{x_1^{a_1-1}x_2^{a_2-1}}\, .
 \end{multline*}
\end{Prop}

\subsection{A general graph} We apply the argument of Section \ref{simgr}
to obtain a contribution formula for a general graph $\Gamma$.

Let $\Gamma\in \mathsf{G}_{g,0}(\PP^4)$ be a decorated graph as defined in Section \ref{locq}. The {\em flags} of $\Gamma$ are the 
half-edges{\footnote{Flags are either half-edges or markings.}}. Let $\mathsf{F}$ be the set of flags. 
Let
$$\mathsf{w}: \mathsf{F} \rightarrow \text{Hom}(\T, \com^*)\otimes_{\mathbb{Z}}{\mathbb{Q}}$$
be a fixed assignment of $\T$-weights to each flag.

We first consider the contribution $\text{Cont}_{\Gamma,\mathsf{w}}$ to 
$$\sum_{d\geq 0} q^d (e(\text{Obs})\cap\left[ \overline{Q}_{g}(\PP^4,d)\right]^{\vir})$$
of the $\T$-fixed loci associated $\Gamma$ satisfying
the following property:
the tangent weight on
the first rational component corresponding
to each $f\in \mathsf{F}$ is
exactly given by $\mathsf{w}(f)$.
We have 
\begin{equation}
    \label{s234}
    \text{Cont}_{\Gamma,\mathsf{w}} = \frac{1}{|\text{Aut}(\Gamma)|}
    \sum_{\mathsf{A} \in \ZZ_{> 0}^{\mathsf{F}}} \prod_{v\in \mathsf{V}} \text{Cont}^{\mathsf{A}}_{\Gamma,\mathsf{w}} (v)\prod _{e\in \mathsf{E}} {\text{Cont}}^{\mathsf{A}}_{\Gamma,\mathsf{w}}(e)\, .
\end{equation}
 The terms on the  right side of \eqref{s234} 
require definition:
\begin{enumerate}
\item[$\bullet$] The sum on the right is over 
the set $\ZZ_{> 0}^{\mathsf{F}}$ of
all maps 
$$\mathsf{A}: \mathsf{F} \rightarrow \ZZ_{> 0}$$
corresponding to the sum over $a_1,a_2$ in
Proposition \ref{propsim}.
\item[$\bullet$]
For $v\in \mathsf{V}$ with 
$n$ incident
flags with $\mathsf{w}$-values $(w_1,\ldots,w_n)$ and
$\mathsf{A}$-values
$(a_1,a_2,...,a_n)$, 
\begin{align*}
    \text{Cont}^{\mathsf{A}}_{\Gamma,{\mathsf{w}}}(v)=
    \frac{\ppl
\psi_1^{a_1-1}, \ldots,
\psi_n^{a_n-1}
\, \Big|\, \mathsf{H}_{\mathsf{g}(v)}^{\mathsf{p}(v)}\,     \ppr_{\mathsf{g}(v),n}^{\mathsf{p}(v),0+}}
{w_1^{a_1} \cdots w_n^{a_n}}\, .
\end{align*}
\item[$\bullet$]
For $e\in \mathsf{E}$ with 
assignments $(\mathsf{p}(v_1), \mathsf{p}(v_2))$
for the two associated vertices{\footnote{In case $e$
is self-edge, $v_1=v_2$.}} and 
$\mathsf{w}$-values $(w_1,w_2)$ for the two associated flags,
    $$    
    \text{Cont}_{\Gamma,\mathsf{w}}(e)=
    e^{\frac{\lann1,1\rann^{\mathsf{p}(v_1),0+}_{0,2}}{w_1}}
    e^{\frac{\lann1,1\rann^{\mathsf{p}(v_2),0+}_{0,2}}{w_2}}
    \mathsf{E}^{w_1,w_2}_{\mathsf{p}(v_1),\mathsf{p}(v_2)}\, .$$
\end{enumerate}
The localization formula then yields \eqref{s234}
just as in the simple case of Section \ref{simgr}.

By summing the contribution \eqref{s234} of $\Gamma$ over
all the weight functions $\mathsf{w}$
and applying Lemma \ref{Edge}, we obtain
the following result which generalizes 
Proposition \ref{propsim}.

\begin{Prop}\label{VE} We have
 $$
 \text{\em Cont}_\Gamma
     =\frac{1}{|\text{\em Aut}(\Gamma)|}
     \sum_{\mathsf{A} \in \ZZ_{> 0}^{\mathsf{F}}} \prod_{v\in \mathsf{V}} 
     \text{\em Cont}^{\mathsf{A}}_\Gamma (v)
     \prod_{e\in \mathsf{E}} \text{\em Cont}^{\mathsf{A}}_\Gamma(e)\, ,
 $$
 where the vertex and edge contributions 
 with incident flag $\mathsf{A}$-values $(a_1,\ldots,a_n)$
 and $(b_1,b_2)$ respectively are
 \begin{eqnarray*}
    \text{\em Cont}^{\mathsf{A}}_\Gamma (v)&=&
    \ppl
\psi_1^{a_1-1}, \ldots,
\psi_n^{a_n-1}
\, \Big|\, \mathsf{H}_{\mathsf{g}(v)}^{\mathsf{p}(v)}\,
  \ppr_{\mathsf{g}(v),n}^{\mathsf{p}(v),0+}\,  ,\\
    \text{\em Cont}^{\mathsf{A}}_\Gamma(e)
    &=&
    (-1)^{b_1+b_2}\left[e^{-\frac{\lann1,1\rann^{\mathsf{p}(v_1),0+}_{0,2}}{x_1}}
       e^{-\frac{\lann1,1\rann^{\mathsf{p}(v_2),0+}_{0,2}}{x_2}}e_i\left(\overline{\mathds{V}}_{ij}-\frac{1}{e_i(x_1+x_2)}\right)e_j\right]_{x_1^{b_1-1}x_2^{b_2-1}}\, ,
 \end{eqnarray*}
where $\mathsf{p}(v_1)=p_i$ and $\mathsf{p}(v_2)=p_j$ in the second equation. 
\end{Prop}

\subsection{Legs} 
Let $\Gamma \in \mathsf{G}_{g,n}(\PP^4)$ be a decorated graph
with markings. While no markings are needed to define the
stable quotient invariants of formal quintic, the contributions
of decorated graphs with markings will appear in the
proof of the holomorphic anomaly equation.
The formula for the contribution $\text{Cont}_\Gamma(H^{k_1},\ldots,H^{k_n})$
of $\Gamma$ to 
\begin{align*}
    \sum_{d\ge 0}q^d \prod_{j=0}^n \text{ev}^*(H^{k_j})\cdot e(\text{Obs})\cap\left[ \overline{Q}_{g}(\PP^4,d)\right]^{\vir}
\end{align*}
is given by the following result.
\begin{Prop}\label{VEL} We have
 \begin{multline*}
 \text{\em Cont}_\Gamma(H^{k_1},\ldots,H^{k_n})
     =\\\frac{1}{|\text{\em Aut}(\Gamma)|}
     \sum_{\mathsf{A} \in \ZZ_{>0}^{\mathsf{F}}} \prod_{v\in \mathsf{V}} 
     \text{\em Cont}^{\mathsf{A}}_\Gamma (v)
     \prod_{e\in \mathsf{E}} \text{\em Cont}^{\mathsf{A}}_\Gamma(e)
     \prod_{l\in \mathsf{L}} \text{\em Cont}^{\mathsf{A}}_\Gamma(l)\, ,
 \end{multline*}
 where the leg contribution 
 is 
 \begin{eqnarray*}
     \text{\em Cont}^{\mathsf{A}}_\Gamma(l)
    &=&
    (-1)^{\mathsf{A}(l)-1}\left[e^{-\frac{\lann1,1\rann^{\mathsf{p}(l),0+}_{0,2}}{z}}
       \overline{\mathds{S}}_{\mathsf{p}(l)}(H^{k_l})\right]_{z^{\mathsf{A}(l)-1}}\, .
 \end{eqnarray*}
The vertex and edge contributions are same as before.
\end{Prop}

The proof of Proposition \ref{VEL} 
follows the vertex and edge analysis. We leave the
details as an exercise for the reader.
The parallel statement for Gromov-Witten theory
can be found in \cite{Elliptic, SS,Book}.

\section{Vertices, edges, and legs} \label{svel}
\subsection{Overview}
Following the analysis of 
$K\proj^2$ in \cite[Section 6]{LP1}
which uses results of Givental \cite{Elliptic,SS,Book} and the  wall-crossing  of \cite{CKg}, we calculate here 
the vertex and edge contributions
 in terms of the function $R_k$ of Section \ref{furcalc}.

\subsection{Calculations in genus 0}
We follow the notation introduced in Section \ref{intmg}. Recall
the series
$$T(c)=t_0 +t_1 c+t_2 c^2+\ldots\, .$$ 

\begin{Prop} {\em (Givental \cite{Elliptic,SS,Book})} For $n\geq 3$, we have
\begin{multline*}
\lann
 1,\ldots,1\rann_{0,n}^{p_i,\infty} = \\
 (\sqrt{\Delta_i})^{2g-2+n}\left(\sum_{k\geq 0}\frac{1}{k!}\int_{\overline{M}_{0,n+k}}T(\psi_{n+1})\cdots T(\psi_{n+k})\right)\Big|_{t_0=0,t_1=0,t_{j\ge 2}=(-1)^j\frac{Q_{j-1}}{\lambda_i^{j-1}}}
 \end{multline*}
 where the functions $\sqrt{\Delta_i}$, $Q_l$ are defined by 
 \begin{align*}
     \overline{\mathds{S}}^{\infty}_i(1) =  e_i \Big\langle \Big\langle  \frac{\phi _i}{z-\psi} , 1 
\Big\rangle \Big\rangle _{0, 2}^{p_i,\infty}=\frac{e^{\frac{\lann1,1\rann^{p_i,\infty}_{0,2}}{z}}}{\sqrt{\Delta_i}}
\left( 1+\sum_{l=1}^\infty Q_l \left(\frac{z}{\lambda_i}\right)^{l}\right)\, .
 \end{align*}
\end{Prop}

From \eqref{VS} and Proposition \ref{WC}, we have
\begin{align*}
    \lann1,1\rann^{p_i,\infty}_{0,2}=\mu \lambda_i\,,\\
    \sqrt{\Delta_i}=\frac{C_0}{R_0}\,,\\
    Q_k=\frac{R_k}{R_0}\,.
\end{align*}

Using Proposition \ref{WC} again, we have proven the following result.

\begin{Prop}\label{q2q2} For $n\geq 3$, we have \label{zaa3}
\begin{multline*}
\lann
 1,\ldots,1\rann_{0,n}^{p_i,0+} = \\
 R_0^{2-n}\left(\sum_{k\geq 0}\frac{1}{k!}\int_{\overline{M}_{0,n+k}}T(\psi_{n+1})\cdots T(\psi_{n+k})\right)\Big|_{t_0=0,t_1=0,t_{j\ge 2}=(-1)^j\frac{R_{j-1}}{\lambda_i^{j-1}R_0}}\, .
 \end{multline*}
\end{Prop}

 Proposition \ref{q2q2} immediately implies the evaluation
\begin{equation} \label{fxxf}
\lann
 1,1,1\rann_{0,3}^{p_i,0+}=\frac{1}{R_0}\, .
 \end{equation}
Another simple consequence of Proposition \ref{zaa3} is the following 
 basic property.
\begin{Cor}\label{Poly} For $n\geq 3$, we have
 $
 \lann
 1,\ldots,1\rann_{0,n}^{p_i,0+} \in \CC[R_0^{\pm 1},R_1,R_2,...][\lambda_i^{-1}]
 $.
\end{Cor}

\subsection{Vertex and edge analysis}
By Proposition \ref{VE}, we have decomposition of the
contribution to $\Gamma\in \mathsf{G}_{g}(\PP^4)$ to
the stable quotient theory of 
formal quintic 
into vertex terms and edge terms
$$
 \text{Cont}_\Gamma
     =\frac{1}{|\text{Aut}(\Gamma)|}
     \sum_{\mathsf{A} \in \ZZ_{> 0}^{\mathsf{F}}} \prod_{v\in \mathsf{V}} 
     \text{Cont}^{\mathsf{A}}_\Gamma (v)
     \prod_{e\in \mathsf{E}} \text{Cont}^{\mathsf{A}}_\Gamma(e)\, .
 $$


\begin{Lemma}\label{L1} We have
    $\text{\em Cont}^{\mathsf{A}}_\Gamma (v)\in \CC(\lambda_0,\dots,\lambda_4)[L^{\pm1}]$. 
\end{Lemma}

\begin{proof} By Proposition \ref{VE}, 
$$\text{Cont}^{\mathsf{A}}_\Gamma (v) = 
    \ppl
\psi_1^{a_1-1}, \ldots,
\psi_n^{a_n-1}
\, \Big|\, \mathsf{H}_{\mathsf{g}(v)}^{\mathsf{p}(v)}\,
  \ppr_{\mathsf{g}(v),n}^{\mathsf{p}(v),0+}\, .$$
 The right side of the above
  formula is a polynomial 
  in the variables
$$\frac{1}{\lann 1,1,1\rann^{\mathsf{p}(v),0+}_{0,3}}\ \ \ 
\text{and} \ \ \
 \Big\{ \, \lann 1,\ldots,1\rann^{\mathsf{p}(v),0+}_{0,n}\, |_{t_0=0} \, \Big\}_{n\geq  3}\, $$
 with coefficients in $\mathbb{C}(\lambda_0,\dots,\lambda_4)$.
The Lemma then follows from 
the evaluation \eqref{fxxf}, Corollary \ref{Poly},
and 
Proposition \ref{RPoly}.

Both the positive and the negative powers of ${\lann 1,1,1\rann^{\mathsf{p}(v),0+}_{0,3}}$
are required here, since $R_0^{\pm1}$ occurs in Corollary \ref{Poly}.
\end{proof}

Let $e\in \mathsf{E}$ be an edge connecting the $\T$-fixed points $p_i, p_j \in \PP^4$. Let
the $\mathsf{A}$-values of the respective
half-edges be $(k,l)$.

\begin{Lemma}\label{L2} We have
 $\text{\em Cont}^{\mathsf{A}}_\Gamma(e) \in \CC(\lambda_0,\dots,\lambda_4)[L^{\pm1},\mathcal{X},\mathcal{X}_1,\mathcal{X}_2,\mathcal{Y}]$ and
 \begin{enumerate}
 \item[$\bullet$]
 the degree of $\text{\em Cont}^{\mathsf{A}}_\Gamma(e)$ with respect to $\mathcal{Y}$ is $1$,
 \item[$\bullet$]
 the coefficient of $\mathcal{Y}$ in 
 $\text{\em Cont}^{\mathsf{A}}_\Gamma(e)$
 is 
 $$(-1)^{k+l+1}\frac{R_{1\, k-1} R_{1\, l-1}}{5 L^3 \lambda_i^{k-2} \lambda_j^{l-2}}\, .$$ 
 \end{enumerate}
\end{Lemma}

\begin{proof}
 By Proposition \ref{VE}, 
$$\text{Cont}^{\mathsf{A}}_\Gamma (e) = 
    (-1)^{k+l}\left[e^{-\frac{\mu \lambda_i}{x}-\frac{\mu \lambda_j}{y}}e_i\left(\overline{\mathds{V}}_{ij}-\frac{\delta_{ij}}{e_i(x+y)}\right)e_j
    \right]_{x^{k-1} y^{l-1}}\, .$$
 Using also the equation
 \begin{align*}
     e_i \overline{\mathds{V}}_{ij} (x, y) e_j  = 
\frac{\sum _{r=0}^4 \overline{\mathds{S}}_i (\phi_r)|_{z=x} \, \overline{\mathds{S}}_j (\phi ^r )|_{z=y}}{x+ y}\, ,
 \end{align*}
we write $\text{Cont}^{\mathsf{A}}_\Gamma (e)$
as
 \begin{align*}
\left[(-1)^{k+l} e^{-\frac{\mu \lambda_i}{x}-\frac{\mu \lambda_j}{y}}\sum_{r=0}^4\overline{\mathds{S}}_i(\phi_r)|_{z=x}\, \overline{\mathds{S}}_j(\phi^r)|_{z=y}
\right]_{x^{k}y^{l-1}-x^{k+1}y^{l-2}+
\ldots +(-1)^{k-1} x^{k+l-1}}
\end{align*}
where the subscript signifies a (signed) sum
of the respective coefficients.
If we substitute the asymptotic expansions \eqref{VS} for
$$\overline{\mathds{S}}_i(1)\, , \ \
\overline{\mathds{S}}_i(H)\, , \ \
\overline{\mathds{S}}_i(H^2)\,  , \ \
\overline{\mathds{S}}_i(H^3)\,  , \ \ 
\overline{\mathds{S}}_i(H^4)
$$ in the above expression, the Lemma follows from Proposition \ref{RPoly2}.
\end{proof}

Similarly, we obtain the following result using Proposition \ref{RPoly3}.
\begin{Lemma}\label{L3}
 We have $\text{\em Cont}^{\mathsf{A}}_{\Gamma}(e) \hspace{-1pt}\in  \CC(\lambda_0,\ldots,\lambda_4)[L^{\pm 1},\mathcal{X},A_2,A_4,A_6]$ and 
 \begin{itemize}
     \item [$\bullet$] the degree of $\text{\em Cont}^{\mathsf{A}}_{\Gamma}$ with respect to $\mathcal{X}$ is 1,
     \item [$\bullet$] the coefficient of $\mathcal{X}$ in $\text{\em Cont}^{\mathsf{A}}_{\Gamma}$ is
     $$(-1)^{k+l+1}\left(\frac{R_{0\, k-1} R_{2\, l-1}}{5 L^3 \lambda_i^{k-1} \lambda_j^{l-3}}+\frac{R_{2\, k-1} R_{0\, l-1}}{5 L^3 \lambda_i^{k-3} \lambda_j^{l-1}}\right)\,.$$
 \end{itemize}
\end{Lemma}

\subsection{Legs}
Using the contribution formula of Proposition \ref{VEL},
\begin{eqnarray*}
     \text{Cont}^{\mathsf{A}}_\Gamma(l)
    &=&
    (-1)^{\mathsf{A}(l)-1}\left[e^{-\frac{\lann1,1\rann^{\mathsf{p}(l),0+}_{0,2}}{z}}
       \overline{\mathds{S}}_{\mathsf{p}(l)}(H^{k_l})\right]_{z^{\mathsf{A}(l)-1}}\, ,
 \end{eqnarray*}
 we easily conclude
 
\begin{enumerate}
 \item[$\bullet$]
 when the insertion at the marking $l$ is $H^0$,$$C_0\cdot \text{Cont}^{\mathsf{A}}_\Gamma(l)\in
\CC(\lambda_0,\dots,\lambda_4)[L^{\pm1}]\, ,$$
 \item[$\bullet$]
 when the insertion at the marking $l$ is $H^1$,
 $$C_0 C_1 \cdot \text{Cont}^{\mathsf{A}}_\Gamma(l)\in
\CC(\lambda_0,\dots,\lambda_4)[L^{\pm1},\mathcal{X}]\, ,$$
\item[$\bullet$]
 when the insertion at the marking $l$ is $H^2$,
 $$C_0C_1C_2 \cdot \text{Cont}^{\mathsf{A}}_\Gamma(l)\in
\CC(\lambda_0,\dots,\lambda_4)[L^{\pm1},\mathcal{X},\mathcal{X}_1,\mathcal{Y}]\, ,$$
\item[$\bullet$]
 when the insertion at the marking $l$ is $H^3$,
 $$C_0C_1C_2C_3 \cdot \text{Cont}^{\mathsf{A}}_\Gamma(l)\in
\CC(\lambda_0,\dots,\lambda_4)[L^{\pm1},\mathcal{X},\mathcal{X}_1,\mathcal{X}_2]\, ,$$
\item[$\bullet$]
 when the insertion at the marking $l$ is $H^4$,
 $$C_0C_1C_2C_3C_4 \cdot \text{Cont}^{\mathsf{A}}_\Gamma(l)\in
\CC(\lambda_0,\dots,\lambda_4)[L^{\pm1},\mathcal{X},\mathcal{X}_1,\mathcal{X}_2]\,.$$
\end{enumerate}

\section{Holomorphic anomaly for the formal quintic}
\label{hafp}

\subsection{Proof of Theorem \ref{ooo5}}

By definition, we have
\begin{align*}\label{ffww}
K_2(q)=& -\frac{1}{L^5}\mathcal{X}\, , \\
A_2(q)=& \frac{1}{L^5}\left(-\frac{1}{5}\mathcal{Y}-\frac{2}{5}\mathcal{X}-\frac{3}{25}\right)\, , \\
A_4(q)=& \frac{1}{L^{10}}\left(-\frac{1}{25}\mathcal{X}^2-\frac{1}{25}\mathcal{X}\mathcal{Y}+\frac{1}{25}\mathcal{X}_1+\frac{2}{25^2} \right)\, , \\
A_6(q)=& \frac{1}{10 \cdot 5^5 L^{15}}\Big(4+125\mathcal{X}_1+50\mathcal{X}(1+10\mathcal{X}_1)\\
&-5L^5(1+10\mathcal{X}+25\mathcal{X}^2+25\mathcal{X}_1)+125\mathcal{X}_2-125\mathcal{X}^2(\mathcal{Y}-1)\Big)\,.
\end{align*}
Hence, statement (i),
$$\widetilde{\mathcal{F}}_g^{\mathsf{B}} (q) \in \mathbb{C}[L^{\pm1}][C_0^{\pm 1},K_2,A_2,A_4,A_6]\, ,$$
follows from Proposition \ref{RPoly2}, \ref{RPoly3}, \ref{VE}
and  Lemmas \ref{L1} - \ref{L3}.

Since 
$$\frac{\partial}{\partial T} = \frac{q}{C_1}\frac{ \partial}{\partial q}\,, $$
statement (ii),
\begin{equation}\label{vvtt}
\frac{\partial^k \widetilde{\mathcal{F}}_g^{\mathsf{B}}}{\partial T^k}(q) \in \mathbb{C}[L^{\pm1}][C_0^{\pm 1},C_1^{-1},K_2,A_2,A_4,A_6]\, ,
\end{equation}
follows since the ring
$$\mathbb{C}[L^{\pm1}][C_0^{\pm 1},C_1^{-1},K_2,A_2,A_4,A_6]=\mathbb{C}[L^{\pm1}][C_0^{\pm 1},C_1^{-1},\mathcal{X},\mathcal{X}_1,\mathcal{X}_2,\mathcal{Y}]$$
is closed under the action of the differential operator $$\DD=q\frac{\partial}{\partial q}\, $$
by \eqref{drule}.
The degree of $C_1^{-1}$ in \eqref{vvtt}
is $1$ which yields statement (iii).
\qed

\begin{Rmk}\label{C0C1}
 {\em The proof of Theorem \ref{ooo5}  
 actually yields:}
 $$C_0^{2-2g}\cdot \widetilde{\mathcal{F}}^{\mathsf{B}}_g\,, \ \ C_0^{2-2g} C_1^k \cdot \frac{\partial^k \widetilde{\mathcal{F}}^{\mathsf{B}}_g}
 {\partial T^k} \,\in\, \CC[L^{\pm 1}][A_2,A_4,A_6,K_2]\,.$$
\end{Rmk}

\subsection{Proof of Theorem \ref{ttt5}: first equation}
\label{prttt}

Let $\Gamma \in \mathsf{G}_{g}(\PP^4)$ be a decorated graph. Let us fix an edge $f\in\mathsf{E}(\Gamma)$:
\begin{enumerate}
\item[$\bullet$] if $\Gamma$ is connected after 
deleting $f$, denote the resulting graph by $$\Gamma^0_f\in \mathsf{G}_{g-1,2}(\PP^4)\, ,$$
\item[$\bullet \bullet$] if $\Gamma$ is disconnected after deleting $f$, denote the resulting two graphs by $$\Gamma^1_f\in \mathsf{G}_{g_1,1}(\PP^4) \ \ \ 
\text{and}\ \ \  \Gamma^2_f\in \mathsf{G}_{g_2,1}(\PP^4)$$
where $g=g_1+g_2$.
\end{enumerate}
There is no canonical
order for the 2 new markings. 
We will always sum over the 2 labellings. So more precisely, the graph
$\Gamma^0_f$ in case $\bullet$
should be viewed as sum
of 2 graphs
$$\Gamma^0_{f,(1,2)} +
\Gamma^0_{f,(2,1)}\, .$$
Similarly, in case $\bullet\bullet$,
we will sum over the ordering of $g_1$ and $g_2$. As usually, the summation
will be later compensated by a factor of
$\frac{1}{2}$ in the formulas.

By Proposition \ref{VE}, we have
the following formula for the contribution 
of the graph $\Gamma$ to the 
theory of the formal quintic,
 $$
 \text{Cont}_\Gamma
     =\frac{1}{|\text{Aut}(\Gamma)|}
     \sum_{\mathsf{A} \in \ZZ_{\ge 0}^{\mathsf{F}}} \prod_{v\in \mathsf{V}} 
     \text{Cont}^{\mathsf{A}}_\Gamma (v)
     \prod_{e\in \mathsf{E}} \text{Cont}^{\mathsf{A}}_\Gamma(e)\, .
 $$


Let $f$ connect the $\T$-fixed points $p_i, p_j \in \PP^4$. Let
the $\mathsf{A}$-values of the respective
half-edges be $(k,l)$. By Lemma \ref{L2}, we have
\begin{equation}\label{Coeff}
\frac{\partial \text{Cont}^{\mathsf{A}}_\Gamma(f)}{\partial \mathcal{Y}} =(-1)^{k+l+1}\frac{R_{1\, k-1} R_{1\, l-1}}{5 L^3 \lambda_i^{k-2} \lambda_j^{l-2}}\, .
\end{equation}

\noindent $\bullet$ If $\Gamma$ is connected after deleting $f$, we have
\begin{multline*}
\frac{1}{|\text{Aut}(\Gamma)|}
     \sum_{\mathsf{A} \in \ZZ_{\ge 0}^{\mathsf{F}}} 
     \left(-\frac{5L^5}{C_0^2C^2_1}\right)
     \frac{\partial {\text{Cont}}^{\mathsf{A}}_\Gamma(f)}{\partial \mathcal{Y}} 
     \prod_{v\in \mathsf{V}} 
     \text{Cont}^{\mathsf{A}}_\Gamma (v)
     \prod_{e\in \mathsf{E},\, e\neq f} \text{Cont}^{\mathsf{A}}_\Gamma(e) \\=
\text{Cont}_{\Gamma^0_f}(H,H) \, .
\end{multline*}
The derivation is simply by using \eqref{Coeff} on the left
and Proposition \ref{VEL} on the right.

\vspace{5pt}
\noindent $\bullet\bullet$
If $\Gamma$ is disconnected after deleting $f$, we obtain
\begin{multline*}
\frac{1}{|\text{Aut}(\Gamma)|}
     \sum_{\mathsf{A} \in \ZZ_{\ge 0}^{\mathsf{F}}} 
     \left(-\frac{5L^5}{C_0^2C^2_1}\right)
     \frac{\partial {\text{Cont}}^{\mathsf{A}}_\Gamma(f)}{\partial \mathcal{Y}} 
     \prod_{v\in \mathsf{V}} 
     \text{Cont}^{\mathsf{A}}_\Gamma (v)
     \prod_{e\in \mathsf{E},\, e\neq f} \text{Cont}^{\mathsf{A}}_\Gamma(e)\\
=\text{Cont}_{\Gamma^1_f}(H) \,
\text{Cont}_{\Gamma^2_f}(H)\, 
\end{multline*}
by the same method. 

By combining the above two equations for all 
the edges of all the graphs $\Gamma\in \mathsf{G}_g(\PP^4)$
and using the vanishing
\begin{align*}
\frac{\partial {\text{Cont}}^{\mathsf{A}}_\Gamma(v)}{\partial \mathcal{Y}}=0
\end{align*}
of Lemma \ref{L1}, we obtain
\begin{multline}\label{greww}
\left(-\frac{L^5}{5C_0^2C^2_1}\right) \frac{\partial}{\partial \mathcal{Y}} 
 \lan  \ran^{\mathsf{SQ}}_{g,0}= \frac{1}{2}\sum_{i=1}^{g-1} \lan H\ran^{\mathsf{SQ}}_{g-i,1}
\lan H \ran^{\mathsf{SQ}}_{i,1} + \frac{1}{2} \lan H,H\ran^{\mathsf{SQ}}_{g-1,2}\, .
\end{multline}
We have followed here the notation of Section \ref{holp}. 

By definition of $A_2, A_4, A_6$, we have following equations.
\begin{align*}
    &\left(\frac{1}{C_1^2}\frac{\partial}{\partial A_2}-\frac{K_2}{5C_1^2}\frac{\partial}{\partial A_4}+\frac{K_2^2}{50C_1^2}\frac{\partial}{\partial A_6}\right)\mathcal{Y}=-5L^5\,,\\
    &\left(\frac{1}{C_1^2}\frac{\partial}{\partial A_2}-\frac{K_2}{5C_1^2}\frac{\partial}{\partial A_4}+\frac{K_2^2}{50C_1^2}\frac{\partial}{\partial A_6}\right)\mathcal{X}_1=0\,,\\
    &\left(\frac{1}{C_1^2}\frac{\partial}{\partial A_2}-\frac{K_2}{5C_1^2}\frac{\partial}{\partial A_4}+\frac{K_2^2}{50C_1^2}\frac{\partial}{\partial A_6}\right)\mathcal{X}_2=0\,.
\end{align*}
Since $I_0^{2g-2}\lan\ran^{\mathsf{SQ}}_g=\widetilde{\mathcal{F}}_g^{\mathsf{B}}$, the left side of \eqref{greww} after multiplication by $I_0^{2g-2}$  is, by the chain rule,
$$\frac{1}{C_0^2C_1^2}\frac{\partial \widetilde{\mathcal{F}}^{\mathsf{B}}_g}{\partial A_2}-\frac{K_2}{5C_0^2C_1^2}\frac{\partial \widetilde{\mathcal{F}}^{\mathsf{B}}_g}{\partial A_4}+\frac{K_2^2}{50C_0^2C_1^2}\frac{\partial \widetilde{\mathcal{F}}^{\mathsf{B}}_g}{\partial A_6}\in \CC[L^{\pm 1}][C_0^{\pm 1},C_1^{-1},K_2,A_2,A_4,A_6]\,.$$

On the right side of \eqref{greww}, we have 
$$ I_0^{2(g-i)-2+1}\lan H  \ran^{\mathsf{SQ}}_{g-i,1}\, =\, \widetilde{\mathcal{F}}_{g-i,1}^{\mathsf{B}}(q)\, =\,  \widetilde{\mathcal{F}}^{\mathsf{GW}}_{g-i,1}(Q(q))\, ,$$
where the first equality is by definition and the second is by
wall-crossing \eqref{jjed}. Then,
$$\widetilde{\mathcal{F}}^{\mathsf{GW}}_{g-i,1}(Q(q))\ = \ \frac{\partial \widetilde{\mathcal{F}}^{\mathsf{GW}}_{g-i}}{\partial T}(Q(q)) \ =\ 
\frac{\partial\widetilde{\mathcal{F}}^{\mathsf{B}}_{g-i}}{\partial T}(q) 
$$
where the first equality is by the divisor equation in
Gromov-Witten theory and the second is again by wall-crossing
\eqref{jjed}. So we conclude
\begin{equation}\label{aa234}
 I_0^{2(g-i)-2+1}\lan H  \ran^{\mathsf{SQ}}_{g-i,1} =\frac{\partial\widetilde{\mathcal{F}}^{\mathsf{B}}_{g-i}}{\partial T}(q)\, \in \mathbb{C}[[q]]\, .
 \end{equation}
Similarly, we obtain
\begin{eqnarray}\label{bb234}
 I_0^{2(g-i)-2+1} \lan H  \ran^{\mathsf{SQ}}_{i,1} &=&\frac{\partial\widetilde{\mathcal{F}}^{\mathsf{B}}_{i}}{\partial T}(q)\, 
\, \in \mathbb{C}[[q]]\, ,
\\ \label{cc234}
 I_0^{2(g-1)-2+2} \lan H,H  \ran^{\mathsf{SQ}}_{g-1,2} &=&\frac{\partial^2\widetilde{\mathcal{F}}^{\mathsf{B}}_{g-1}}{\partial T^2}(q)\,
\, \in \mathbb{C}[[q]]\, .
\end{eqnarray}
The above equations transform \eqref{greww}, after multiplication by $I_0^{2g-2}$, to 
exactly the first holomorphic anomaly equation of Theorem \ref{ttt5},
$$\frac{1}{C_0^2C_1^2}\frac{\partial \widetilde{\mathcal{F}}_g^{\mathsf{B}}}{\partial{A_2}}-\frac{1}{5C_0^2C_1^2}\frac{\partial \widetilde{\mathcal{F}}_g^{\mathsf{B}}}{\partial{A_4}}K_2+\frac{1}{50C_0^2C_1^2}\frac{\partial \widetilde{\mathcal{F}}_g^{\mathsf{B}}}{\partial{A_6}}K_2^2
= \frac{1}{2}\sum_{i=1}^{g-1} 
\frac{\partial \widetilde{\mathcal{F}}_{g-i}^{\mathsf{B}}}{\partial{T}}
\frac{\partial \widetilde{\mathcal{F}}_i^{\mathsf{B}}}{\partial{T}}
+
\frac{1}{2}
\frac{\partial^2 \widetilde{\mathcal{F}}_{g-1}^{\mathsf{B}}}{\partial{T}^2}\,
$$
as an equality in $\mathbb{C}[[q]]$.

In order to lift
the first
holomorphic anomaly
equation to the ring
$$\mathbb{C}[L^{\pm1}][A_2,A_4,A_6,C_0^{\pm1},C_1^{-1},K_2]\, ,$$
we must lift
the equalities
\eqref{aa234}-\eqref{cc234}. The proof
is identical to 
the parallel lifting
for $K\proj^2$ given
in \cite[Section 7.3]{LP1}.

\subsection{Proof of Theorem \ref{ttt5}: second equation}


By Proposition \ref{VE}, we have
the following formula for the contribution 
of the graph $\Gamma$ to the stable quotient
theory of formal quintic,
 $$
 \text{Cont}_\Gamma
     =\frac{1}{|\text{Aut}(\Gamma)|}
     \sum_{\mathsf{A} \in \ZZ_{\ge 0}^{\mathsf{F}}} \prod_{v\in \mathsf{V}} 
     \text{Cont}^{\mathsf{A}}_\Gamma (v)
     \prod_{e\in \mathsf{E}} \text{Cont}^{\mathsf{A}}_\Gamma(e)\, .
 $$


Let $f$ connect the $\T$-fixed points $p_i, p_j \in \PP^4$. Let
the $\mathsf{A}$-values of the respective
half-edges be $(k,l)$. By Lemma \ref{L3}, we have
\begin{equation}\label{Coeff2}
\frac{\partial \text{Cont}^{\mathsf{A}}_\Gamma(f)}{\partial \mathcal{X}} =(-1)^{k+l+1}\left(\frac{R_{0\, k-1} R_{2\, l-1}}{5 L^3 \lambda_i^{k-1} \lambda_j^{l-3}}+\frac{R_{2\, k-1} R_{0\, l-1}}{5 L^3 \lambda_i^{k-3} \lambda_j^{l-1}}\right)\, .
\end{equation}

\noindent $\bullet$ If $\Gamma$ is connected after deleting $f$, we have
\begin{multline*}
\frac{1}{|\text{Aut}(\Gamma)|}
     \sum_{\mathsf{A} \in \ZZ_{\ge 0}^{\mathsf{F}}} 
     \left(-\frac{5L^5}{C^2_1}\right)
     \frac{\partial {\text{Cont}}^{\mathsf{A}}_\Gamma(f)}{\partial \mathcal{X}} 
     \prod_{v\in \mathsf{V}} 
     \text{Cont}^{\mathsf{A}}_\Gamma (v)
     \prod_{e\in \mathsf{E},\, e\neq f} \text{Cont}^{\mathsf{A}}_\Gamma(e) \\=
\text{Cont}_{\Gamma^0_f}(1,H^2)\, .
\end{multline*}
The derivation is simply by using \eqref{Coeff2} on the left
and Proposition \ref{VEL} on the right.

\vspace{5pt}
\noindent $\bullet\bullet$
If $\Gamma$ is disconnected after deleting $f$, we obtain
\begin{multline*}
\frac{1}{|\text{Aut}(\Gamma)|}
     \sum_{\mathsf{A} \in \ZZ_{\ge 0}^{\mathsf{F}}} 
     \left(-\frac{5L^5}{C^2_1}\right)
     \frac{\partial {\text{Cont}}^{\mathsf{A}}_\Gamma(f)}{\partial \mathcal{X}} 
     \prod_{v\in \mathsf{V}} 
     \text{Cont}^{\mathsf{A}}_\Gamma (v)
     \prod_{e\in \mathsf{E},\, e\neq f} \text{Cont}^{\mathsf{A}}_\Gamma(e)\\
=\text{Cont}_{\Gamma^1_f}(1) \,
\text{Cont}_{\Gamma^2_f}(H^2)\, 
\end{multline*}
by the same method. 

By combining the above two equations for all 
the edges of all the graphs $\Gamma\in \mathsf{G}_g(\PP^4)$
and using the vanishing
\begin{align*}
\frac{\partial {\text{Cont}}^{\mathsf{A}}_\Gamma(v)}{\partial \mathcal{X}}=0
\end{align*}
of Lemma \ref{L1}, we obtain
\begin{multline}\label{greww2}
\left(-\frac{5L^5}{C^2_1}\right) \frac{\partial}{\partial \mathcal{X}} 
 \lan  \ran^{\mathsf{SQ}}_{g,0}= \sum_{i=1}^{g-1} \lan 1\ran^{\mathsf{SQ}}_{g-i,1}
\lan H^2 \ran^{\mathsf{SQ}}_{i,1} + \lan 1,H^2\ran^{\mathsf{SQ}}_{g-1,2} =0\, .
\end{multline}

The second equality in the above equations follow from the string equation for formal stable quotient invariants. Since 
$$K_2=-\frac{1}{L^5}\mathcal{X}\,,$$ the equation \eqref{greww2} is equivalent to the second holomorphic
anomaly equation of Theorem \ref{ttt5} as an equality
in $\com[[q]]$.

In order to lift
the second
holomorphic anomaly
equation to the ring
$$\mathbb{C}[L^{\pm1}][A_2,A_4,A_6,C_0^{\pm1},C_1^{-1},K_2]\, ,$$
we must lift
the equalities
\begin{eqnarray*}
\lan 1\ran^{\mathsf{SQ}}_{g-i,1}& = & 0 \, , \\
\lan 1,H^2\ran^{\mathsf{SQ}}_{g-1,2} & = & 0 \, .
\end{eqnarray*}
 The proof
follows from the properties of the unit $1$
in a CohFT. Specifically, the method
of the proof of \cite[Proposition 2.12]{PPZ}
is used. We leave the details to the reader.

\subsection{Genus one invariants}
We do not study the genus 1 unpointed series $\widetilde{\mathcal{F}}^{\mathsf{B}}_1(q)$ in the paper, so we take
\begin{eqnarray*}
 I_0\cdot\lan H  \ran^{\mathsf{SQ}}_{1,1} &=&\frac{\partial\widetilde{\mathcal{F}}^{\mathsf{B}}_{1}}{\partial T}\, ,\\
 I_0^2\cdot\lan H,H  \ran^{\mathsf{SQ}}_{1,2} &=&\frac{\partial^2\widetilde{\mathcal{F}}^{\mathsf{B}}_{1}}{\partial T^2}\, .
\end{eqnarray*}
as definitions of the right side in the genus 1 case.
There is no difficulty in calculating these series explicitly
using Proposition \ref{VEL},

\begin{eqnarray*}
\frac{\partial \widetilde{\mathcal{F}}_1^{\mathsf{B}}}{\partial{T}} &= & \frac{ L^5}{C_1}\left(\frac{1}{2}A_2+\frac{19}{24}K_2+\frac{1}{12}-\frac{19}{120}\frac{1}{L^5}\right),\\
\frac{\partial^2 \widetilde{\mathcal{F}}_1^{\mathsf{B}}}{\partial{T}^2}& = & \frac{1}{C_1}\mathsf{D}\left(\frac{L^5}{C_1}\left(\frac{1}{2}A_2+\frac{19}{24}K_2+\frac{1}{12}-\frac{19}{120}\frac{1}{L^5}\right)\right)
.
\end{eqnarray*}

\subsection{Bounding the degree.}\label{Btd}
For the holomorphic anomaly equation
for $K\proj^2$, the integration constants
can be bounded \cite[Section 7.5]{LP1}.
A parallel result hold for the formal quintic.

The degrees in $L$ of the term of 
$$\widetilde{\mathcal{F}}^{\mathsf{SQ}}_g\in\CC[L^{\pm 1}][A_2,A_4,A_6,K_2]$$
for formal quintic always fall in the range
\begin{align}\label{DB}
    [15-15g,10g-10]
\end{align}
In particular, the constant (in $A_2,A_4,A_6,K_2$) term of $\widetilde{\mathcal{F}}^{\mathsf{SQ}}_g$ missed by the holomorphic anomaly equation for formal quintic is a Laurent polynomial in $L$ with degree in the range \eqref{DB}. The bound \eqref{DB} is a consequence of Proposition \ref{VE}, the vertex and edge analysis of Section \ref{svel}, and the following result.

\begin{Lemma}
The degrees in $L$ of $R_{ip}$ fall in the range
\begin{align*}
    [-i,4p+1]\,.
\end{align*}
\end{Lemma}

\begin{proof}
 The proof for the functions $R_{0p}$ follows from the arguments of \cite{ZaZi}. The proof for the other $R_{ip}$ follows from Lemma \ref{S1}.
\end{proof}

\end{document}